\documentclass[12pt]{amsart}

\usepackage{amsthm}
\newtheorem{theorem}{Theorem}[section]
\newtheorem{lemma}[theorem]{Lemma}
\newtheorem{proposition}[theorem]{Proposition}
\newtheorem{corollary}[theorem]{Corollary}
\theoremstyle{definition}
\newtheorem{definition}[theorem]{Definition}

\newtheorem{example}[theorem]{Example}
\theoremstyle{remark}
\newtheorem{remark}[theorem]{Remark}

\counterwithin*{section}{part}

\usepackage{amssymb}
\usepackage{empheq}
\usepackage{stmaryrd}
\usepackage{url}
\usepackage{esvect}
\usepackage{float}

\DeclareMathOperator{\N}{\mathbb{N}}

\DeclareMathOperator{\R}{\mathbb{R}}
\DeclareMathOperator*{\argmin}{\arg\!\min} 

\newcommand{\CAT}{{\rm{CAT}(0)}}
\newcommand{\norm}[1]{\lVert #1 \rVert}
\newcommand{\svv}[4]{\left\langle \vv{#1#2},\vv{#3#4}\right\rangle}
\newcommand{\lowerbound}{\frac{1}{\gamma - 8 \eta}}
\newcommand{\upperbound}{\frac{1}{4\eta}}

\usepackage[left=2.0cm,%
                right=2.0cm,%
                top=2.5cm,%
                bottom=3.5cm,%
                headheight=12pt,%
                a4paper]{geometry}%

\begin{document}

\title[On strongly quasiconvex pseudomonotone equilibrium problems]{On strongly quasiconvex pseudomonotone equilibrium problems in Hadamard spaces}

\author[L.M. Despr\'es and N. Pischke]{Luisa Marie Despr\'es${}^{\MakeLowercase a}$ and Nicholas Pischke${}^{\MakeLowercase b}$}
\date{\today}
\maketitle
\vspace*{-5mm}
\begin{center}
{\scriptsize 
${}^a$ Department of Mathematics, Technische Universit\"at Darmstadt,\\
Schlossgartenstra\ss{}e 7, 64289 Darmstadt, Germany,\\ 
${}^b$ Department of Computer Science, University of Bath,\\
Claverton Down, Bath, BA2 7AY, United Kingdom,\\
E-mails: luisa.despres@stud.tu-darmstadt.de, nnp39@bath.ac.uk}
\end{center}

\maketitle
\begin{abstract}
We study two proximal point type methods for finding equilibrium points of pseudomonotone and strongly quasiconvex bifunctions. Extending results by A.\ Iusem and F.\ Lara, we prove the strong convergence of these methods over general complete geodesic metric spaces of nonpositive curvature, so-called Hadamard spaces. Our arguments are quite elementary and in particular effective, yielding sublinear non-asymptotic guarantees for the distance of the iterates towards the solution. These quantitative results are novel even in the context of Euclidean spaces, the original setting of the work by Iusem and Lara, and the simplicity of our arguments allows us to either weaken or even fully discharge various assumptions featuring in this previous work. We also provide an existence result for solutions of equilibrium problems generated by suitably semicontinuous, pseudomonotone and strongly quasiconvex bifunctions over general Hadamard spaces, and derive from this that every lower-semicontinuous strongly quasiconvex function over a Hadamard space has a minimizer, answering a question of the second author.
\end{abstract}

\section{Introduction}

\subsection{Background and motivation}

Equilibrium problems, that is problems of the form 
\begin{equation*}
\text{find } x^{*}\in K\text{ such that } f(x^{*},y) \geq 0 \text{ for all } y\in K,
\end{equation*}
where $f: K \times K \rightarrow \R$ is a suitable bifunction (e.g.\ satisfying appropriate continuity, convexity, and monotonicity conditions) and $K \subseteq X$ is a constraint set (commonly chosen closed and convex) over some space $X$, play a fundamental role in modern applied mathematics. Going back to the work of Fan \cite{Fan72}, convex and more recently also in particular nonconvex equilibrium problems have been studied deeply. We refer e.g.\ to  \cite{Blum1994,IusemKassaySosa09,IusemSosa03,Lopez12, Oettli97}, and respectively to \cite{Cotrina18,Flores01,IusemLara19}, among many others.

In the present paper, we study different approximation methods for such equilibrium problems. These methods, in one way or the other, are derived from the seminal proximal point algorithm originating in the work of Martinet \cite{Martinet70}, Rockafellar \cite{RockTyr76} as well as Brezis and Lions \cite{BrezisLions78}, originally devised to solve convex minimization problems (and, in a generalized formulation, inclusion problems of maximally monotone operators). Concretely, over a Hilbert space $(X,\langle\cdot,\cdot\rangle)$ with norm $\norm{\cdot}$, and given a lower-semicontinuous convex function $h:K\to\mathbb{R}$ over a closed and convex set $K\subseteq X$, the proximal point algorithm iteratively compute points $\{x_k\}_k$ via
\[
x_{k+1}:=\mathrm{Prox}_{\beta_k h}(x_k)
\]
for a given starting point $x_0\in K$, a sequence $\{\beta_k\}_k\subseteq (0,\infty)$ of parameters satisfying suitable conditions, and where $\mathrm{Prox}_{\beta h}$ is the proximal map of $h$, defined by 
\[
\mathrm{Prox}_{\beta h}(x) :=\argmin_{y\in K}\left\{h(y) + \frac{1}{2\beta}\norm{y-x}^2\right\}.
\]
As $h(\cdot) + \frac{1}{2\beta}\norm{\cdot-x}^2$ is strongly convex whenever $h$ is convex, this so-called proximal subproblem has a unique solution so that $\mathrm{Prox}_{\beta h}(x)$ is a well-defined function. The sequence $\{x_k\}_k$ generated by this algorithm converges weakly to a minimizer for $h$ if there exists one (see e.g.\ \cite{BauschCom2017} for a modern canonical reference). If the convexity condition on $h$ is strengthened to strong convexity, one even obtains strong convergence to the unique minimizer of $h$.

Proximal point type algorithms for equilibrium problems have been considered for convex bifunctions at least since the well-known work of Moudafi \cite{Moudafi99} and Moudafi and Th{\'e}ra \cite{MoudafiThera99}, and have subsequently been extended in various ways (we refer in particular to \cite{BURACHIK2012,IusemSosa10,Khatibzadeh2016,Konnov03,Moudafi03}, just to name a few).

At the core of it, the perhaps most prominent proximal point type algorithm for equilibrium problems (already going back to \cite{Moudafi99,MoudafiThera99}) produces an approximating sequence $\{x_k\}_k$ by using the current iterate $x_k$ to define a regularization of the bifunction \emph{as a whole} via
\[
f_k(x,y) := f(x,y) + \frac{1}{\beta_k} \langle x-x_k,y-x\rangle,
\]
and then solving the equilibrium problem associated with $f_k$ to generate $x_{k+1}$. The set-valued map
\[
J^f_\beta(x):=\left\{z\in K\mid f(z,y)+\frac{1}{\beta} \langle z-x,y-z\rangle\geq 0\text{ for all }y\in K\right\},
\]
associating with every point $x$ and parameter $\beta$ the set of solutions of this corresponding regularized equilibrium problem, is also called the \emph{resolvent of the bifunction} $f$ (see e.g.\ \cite{Blum1994,CombettesHirstoaga2005,MoudafiThera99}) and is single-valued whenever $f$ is monotone as well as convex and lower-semicontinuous in its right argument (cf.\ \cite[Lemma 2.12]{CombettesHirstoaga2005}), so that, using this notion, the regularized selection method sketched above reduces to the more familiar pattern $x_{k+1}:=J^f_{\beta_k}(x_k)$ of the proximal point algorithm discussed above.

A second type of proximal point algorithm applies a regularization akin to the proximal subproblem only to the second (convex) argument of the bifunction, with the first argument fixed to the current iterate $x_k$, and so defines
\[
x_{k+1}:=\argmin_{y\in K}\left\{f(x_k,y) + \frac{1}{2\beta_k}\norm{y-x_k}^2\right\}=\mathrm{Prox}_{\beta_kf(x_k,\cdot)}(x_k).
\]
In fact, to our knowledge, this second type of proximal point method for equilibrium problems was only recently introduced in the work of Iusem and Lara \cite{IusemLara2021} (see also \cite{HieuDuongThai2021,IusemMohebbi2020} and the very recent work \cite{BoufiFadilRoubi2026}), to solve equilibrium problems with relaxed convexity assumptions, which will also be the main concern of this paper.

Concretely, the work \cite{IusemLara2021} studies the two methods above for bifunctions which exchange the convexity requirement for the right argument of $f$ with \emph{strong quasiconvexity}. The class of quasiconvex functions, that is functions $h:K\to\mathbb{R}$ which satisfy
\[
h((1-\lambda) x+\lambda y)\leq \max\{h(x),h(y)\}
\]
for all $x,y\in K$ and $\lambda\in [0,1]$, widely relaxes convexity and covers many practically interesting functions that fail to be convex. However, as highlighted and discussed in detail in e.g.\ \cite{Lara2022}, this class simultaneously suffers from a range of problems in regards to iterative procedures, in particular the proximal point algorithm. Starting with the work of Lara \cite{Lara2022}, considerable focus has recently been placed on the class of strongly quasiconvex functions, that is functions $h:K\to\mathbb{R}$ which satisfy
\[
h((1-\lambda) x+\lambda y)\leq \max\{h(x),h(y)\}-\lambda(1-\lambda)\frac{\gamma}{2}\norm{x-y}^2
\]
for all $x,y\in K$ and $\lambda\in [0,1]$, for a given modulus $\gamma>0$. This class already goes back to the work of Polyak \cite{Polyak1966} and has various benefits over the class of quasiconvex functions. In particular, minimizers of such functions are unique if they exist and, partly as a consequence, many iterative procedures known from convex settings are still reasonably well-behaved. Indeed, in \cite{Lara2022}, Lara studies the proximal point method for strongly quasiconvex functions and, while the proximal map in this context is generally only set-valued (cf.\ \cite[Remark 6]{Lara2022}) so that the iterative step then takes the form of a selection $x_{k+1} \in \mathrm{Prox}_{\beta_k h}(x_k)$, he proves a convergence theorem over Euclidean space $\R^d$ (cf.\ \cite[Theorem 10]{Lara2022}). It is then for (pseudomonotone and suitably Lipschitz) bifunctions which are only strongly quasiconvex in their second argument that Iusem and Lara in \cite{IusemLara2021} establish the convergence of the two proximal point type methods discussed above over Euclidean space $\R^d$ (cf.\ \cite[Theorem 3.1]{IusemLara2021} and \cite[Corollary 3.2]{IusemLara2021}), building on the previous work of Lara \cite{Lara2022}.

All the while, the class is not too restrictive, containing various nonconvex functions of practical interest such as the Euclidean norm (cf.\ \cite[Theorem 2]{Jov1996}) which is not strongly convex on any bounded convex set, the square root of the Euclidean norm which is strongly quasiconvex on each ball (cf.\ \cite[Theorem 17]{Lara2022}) without being differentiable, or even functions such as $\max\{\sqrt{\norm{\cdot}},\norm{\cdot}^2-k\}$ for $k\in\mathbb{N}$ which are strongly quasiconvex without even being convex (cf.\ \cite[Remark 18]{Lara2022}). We refer to \cite{Lara2022} and the more recent survey \cite{GLM2025} for further discussions on the relationships between strongly quasiconvex functions and various other classes of (non)convex functions.

\subsection{The contributions of the present paper and related work}

In the present paper, we study the two proximal point type methods sketched above for strongly quasiconvex pseudomonotone equilibrium problems over Hadamard spaces, that is complete geodesic metric spaces of nonpositive curvature. Also known as (complete) $\CAT$-spaces, this general class of nonlinear spaces contains and unifies spaces such as (infinite-dimensional) Hilbert spaces, $\mathbb{R}$-trees and Hadamard manifolds, that is complete simply connected Riemannian manifolds of nonpositive sectional curvature, among many more.

While the fundamental concepts of Hadamard spaces can already be dated back to a paper by Wald \cite{Wald36}, the notion reached maturity through the work of Alexandrov \cite{Aleksandrov51,Alexandrov57} (after which such spaces are also called spaces of nonpositive curvature in the sense of Alexandrov). The moniker $\CAT$ derives from the influential work of Gromov \cite{Gromov1987}. We refer to \cite{Bacak2014,Bacak23,BridsonHaefliger1999} for further historical remarks, and in general refer to \cite{BridsonHaefliger1999} for an authoritative treatment of these spaces at large. Next to their traditional role in geometry, Hadamard spaces are a suitable framework for extending many concepts from convex analysis and optimization on linear spaces to nonlinear contexts, as is the topic of this paper, and we refer to \cite{Bacak2014,Bacak23} for discussions in that direction. We provide a technical overview of the key aspects of Hadamard spaces as relevant to the present paper in Section \ref{subsec_Hadmard_spaces} below.

The proximal point method for convex minimization was previously investigated on various subclasses of these spaces by a range of authors, such as Ferreira and Oliveira \cite{FerreiraOliveira2002} on Hadamard manifolds (see also \cite{LiLopezMartinMarquez2009,PapaQuirozOliveira2009}), which culminated in the seminal work of Ba\v{c}\'ak \cite{Bac2013} who proved the weak convergence of the proximal point method in general Hadamard spaces. Extensions of such results to equilibrium problems have also been studied by various authors, such as Colao, L\'opez, Marino and Mart{\'i}n-M{\'a}rquez \cite{COLAO201261} set in Hadamard manifolds or Khatibzadeh and Mohebbi \cite{KHATIBZADEHMOHEBBI2021} as well as Iusem and Mohebbi \cite{IusemMohebbi2020} over Hadamard spaces. Yet another approach over Hadamard manifolds, utilizing Busemann functions, was provided recently by Bento, Cruz Neto, Melo, et al.\ \cite{BentoCruzNetoLopesMeloFilho2024,BentoCruzNetoMelo2022}. In particular, these works outline three distinct approaches for extending the regularization of the full bifunction to nonlinear spaces, all of which we will study in this paper. In fact, to our knowledge, in two of these cases (modeled after the works \cite{COLAO201261} and \cite{BentoCruzNetoLopesMeloFilho2024,BentoCruzNetoMelo2022}, respectively), the corresponding regularization is considered in general Hadamard spaces for the first time.

Moving to the setting of quasiconvex functions in such nonlinear contexts, there is similarly a wide array of literature, in particular on proximal point methods. These are however, to our knowledge, largely limited
to the setting of minimization problems over Hadamard manifolds, and we refer in particular to the work of Papa Quiroz and his co-authors \cite{Quiroz24,QuirzoAlexCusiMac20,PapaQuirozOliveira2009,QO2012a,QO2012b}. Work on strongly quasiconvex functions in these contexts is however very limited so far and, to our knowledge, only comprises the recent work \cite{Pischke2025} of the second author where he extends Lara's convergence results from \cite{Lara2022} to the setting of Hadamard spaces.
 
In the present paper, we now follow the approach of \cite{Pischke2025} and, for appropriate metric variants of the methods from \cite{IusemLara2021}, establish their convergence over Hadamard spaces. Already for infinite-dimensional Hilbert spaces, these results seem to be new to the literature. The convergence proofs we provide are, similar to \cite{Pischke2025}, rather different from those given in \cite{IusemLara2021} and in particular quite elementary, for example being void of any (weak) compactness arguments. This elementarity of the convergence proofs given here is not only key for overcoming the obstacles of extending the results from \cite{IusemLara2021} to general Hadamard spaces, but in particular allows us to simultaneously provide quantitative variants of these convergence results. Concretely, in the setting of a Hadamard space $(X,d)$, we provide explicit non-asymptotic guarantees for the distance of the iterates $\{x_k\}_k$ generated by any of the considered methods towards the (in the context of a strongly quasiconvex and pseudomonotone bifunction necessarily unique) solution $x^*$ of the corresponding equilibrium problem. These take the concrete form of 
\[
d(x_k,x^{*})< \frac{C}{\sqrt{k-2}} \text{ for all }k > 2
\]
for a specific (explicitly calculated) constant $C$ which only depends on moduli $\gamma$ and $\eta$ featuring in the strong quasiconvexity and a respective Lipschitz property of the bifunction (similar to \cite{IusemLara2021}) and an upper bound $b^2\geq d^2(x_0,x^*)$ on the initial displacement. Our quantitative estimates are otherwise completely independent of the space, iteration or bifunction generating the equilibrium problem. To our knowledge, these rates are novel even over Euclidean spaces, the original setting of the work of Iusem and Lara \cite{IusemLara2021}. These quantitative results can be found in Theorems \ref{THEOREM2}, \ref{THEOREMa} and \ref{secondQuantThm} below.

In particular, our convergence results are thereby ``strong'', i.e.\ we establish the convergence of the methods relative to the metric and not relative to any notion of weak convergence. This is essentially due to the fact that equilibrium problems for strongly quasiconvex and pseudomonotone bifunctions have a unique solution, and in particular sets such results apart from the ``usual'' proximal point method for (bi)convex functions which, as shown by G\"uler \cite{Gue1991}, generally only converges weakly already in Hilbert spaces.

In the course of our convergence proofs, we in particular further show how some semicontinuity assumptions on the bifunction featuring in \cite{IusemLara2021} can be avoided either conditionally or even unconditionally, generalizing the respective range of bifunctions and clarifying the contributions of these continuity assumptions to the convergence of the methods (see Remarks \ref{firstAlgStrengthening} and \ref{secondAlgStrengthening} later on).

Further, in the context of the second type of method which regularizes the bifunction only in its second argument, we are able to weaken the previous parameter restriction $\{\beta_k\}_k\subseteq (\lowerbound,\upperbound)$ featuring in \cite{IusemLara2021} to $\{\beta_k\}_k\subseteq [\lowerbound,\upperbound]$. While perhaps only a marginal improvement at first sight, this extension of the parameter range is actually crucial for establishing our quantitative results for that method, as discussed in more detail later (see the discussion in Section \ref{subsubsection_4_closed} later on). In particular, this extension requires various additional technicalities compared to the proof of the original result given in \cite{IusemLara2021}.

Also, in the context of the first type of method which regularizes the full bifunction, we highlight some issues in the well-definedness argument of that method provided in \cite{IusemLara2021} which leave it open whether this method is generally well-defined for strongly quasiconvex bifunctions already over Euclidean spaces (see Remark \ref{counter} later on).

As a last contribution, we provide a theorem on the existence of solutions for the equilibrium problem of suitably semicontinuous, pseudomonotone and strongly quasiconvex bifunctions over Hadamard spaces (see Theorem \ref{nonemptyness_thrm} later on), lifting the corresponding result obtained by Iusem and Lara (cf.\ \cite[Proposition 3.1]{IusemLara2021}) from the linear setting to this class of nonlinear spaces. Our argument here in particular relies on combining ideas from \cite{IusemLara2021,Lara2022} with the work of Khatibzadeh and Mohebbi \cite{KHATIBZADEHMOHEBBI2021}, where the existence of equilibrium points over Hadamard spaces is studied for certain (in particular convex) bifunctions. As a consequence, we in particular establish that every lower-semicontinuous strongly quasiconvex function on a Hadamard space has a minimizer (see Corollary \ref{minExists} later on), which answers a question left open in the previous work \cite{Pischke2025} of the second author and lifts a corresponding result of Lara (cf.\ \cite[Corollary 3]{Lara2022}) to this nonlinear setting.

\subsection{Proof mining}

All of the present quantitative results have been derived using methods from proof mining, a program in mathematical logic that applies results from proof theory to areas of core mathematics to classify and extract the computational content of prima facie ``non-computational'' proofs therein. We refer to the seminal monograph \cite{Kohlenbach2008} for a comprehensive overview of this area and to the survey \cite{Kohlenbach18Survey} for an overview of various applications to nonlinear analysis in particular. A detailed discussion on the concrete use of logical methods in the context of the present paper can be found in the master thesis of the first author \cite{Despres2026}. The present paper however is (apart from a few short remarks labeled ``for logicians'') presented without any explicit reference to logic and does not require any such background.

\subsection{Future work}

The results of Lara \cite{Lara2022} as well as Iusem and Lara \cite{IusemLara2021} have since been extended in various directions by Lara and his co-authors (see e.g.\ \cite{AnsariBabuRaju2025,GLM2023,GradLaraMarca24,HadjisavvasLaraMarcavillacaVuong2026,IusemLaraMarcavillacaYen2024,LaraMarca24,LaraMarcavillacaVuong2025}, among others), and so there is a natural range of potential future works stemming from the present paper.

In particular, we want to highlight the extensions of the proximal point algorithm which incorporate inertia terms and over-relaxations, as developed by Grad, Lara and Marcavillaca in \cite{GLM2023} for strongly quasiconvex minimization problems and in \cite{GradLaraMarca24} (see again also \cite{HieuDuongThai2021}) for associated equilibrium problems. A forthcoming paper by the authors (based, similarly to the present paper, on the master thesis of the first author \cite{Despres2026}) will provide analogous (quantitative) results also for these methods, which in particular crucially rely on the approach and methods developed in the present paper.

Further, more speculative future work which we want to highlight includes extensions of the methods studied in the present paper to incorporate Bregman distances over Hadamard manifolds, which could be approached by combining the method of the present work with that of preceding work on Bregman distances and associated proximal point methods in the context of convex bifunctions \cite{BURACHIK2012} and strongly quasiconvex (bi)functions \cite{AnsariBabuRaju2025,LaraMarca24} and with work on Bregman distances over such manifolds \cite{PapaQuirozOliveira2009} (or potentially even work of the second author \cite{Pischke2026} on Bregman distances in Hadamard spaces).

\section{Preliminaries}\label{subsec_Hadmard_spaces}

\subsection{Hadamard spaces}

Let $(X,d)$ be a metric space. A geodesic in $X$ is defined to be an isometry $\gamma:[0,l]\to X$, and we say that $\gamma$ joins $\gamma(0)$ and $\gamma(l)$. A geodesic space $X$ is called (uniquely) geodesic if every two points are joined by a (unique) geodesic. A geodesic ray is an isometry $r:[0,\infty)\to X$, and we say that $r$ issues from $r(0)$. $X$ is said to have the geodesic extension property if for all $x\neq y\in X$, there is a ray $r:[0,\infty)\to X$ issuing from $x$ such that $r(t)=y$ for some $t>0$. A geodesic metric space $(X,d)$ is called a $\CAT$-space if it satisfies 
\[
d^2(\gamma(\lambda l),x)\leq (1-\lambda)d^2(\gamma(0),x)+\lambda d^2(\gamma(l),x)-\lambda(1-\lambda)d^2(\gamma(0),\gamma(l))\tag{CN$^+$}\label{CN}
\]
for all $x\in X$, all $\lambda\in [0,1]$ and all geodesics $\gamma:[0,l]\to X$, an extension of the so-called Bruhat-Tits $\mathrm{CN}$-inequality \cite{BruhatTits1972} to geodesics. Any $\CAT$-space is uniquely geodesic and a complete $\CAT$-space is called a Hadamard space. Over a Hadamard space $(X,d)$, we write $(1-\lambda)x \oplus \lambda y$ for the point $\gamma(\lambda d(x,y))$ on the unique geodesic $\gamma:[0,d(x,y)]\to X$ joining $x$ and $y$. We call a subset $K \subseteq X$ convex, if with $x,y \in K$ and $\lambda \in [0,1]$ also $(1-\lambda)x \oplus \lambda y \in K$. We denote the convex hull of a set $A\subseteq X$ by $\mathrm{conv}(A)$.

In this paper, we also rely on an equivalent characterization of Hadamard spaces through a metric version of the Cauchy-Schwarz inequality for the so-called quasi-inner product due to Berg and Nikolaev \cite{BergNikolaev2008}, that is the function
\[
\langle\vv{xy},\vv{uv}\rangle:= \frac{1}{2}\big(d^2(x,v)+d^2(y,u)-d^2(x,u)-d^2(y,v)\big).
\]
As shown in \cite[Proposition 14]{BergNikolaev2008}, this function is the unique map $\langle \cdot, \cdot \rangle : X^2 \times X^2 \rightarrow \R $ satisfying the following conditions for any $x,v,u,v\in X$: $\langle \vv{xy},\vv{xy}\rangle = d^2(x,y)$; $\langle \vv{xy},\vv{uv}\rangle = \langle \vv{uv},\vv{xy}\rangle $;  $\langle \vv{xy},\vv{uv}\rangle = -\langle \vv{yx},\vv{uv}\rangle $; $\langle \vv{xy},\vv{uv}\rangle + \langle \vv{xy},\vv{vw}\rangle = \langle \vv{xy},\vv{uw}\rangle $. Now, as established in \cite[Corollary 3]{BergNikolaev2008}, a geodesic metric space $(X,d)$ is a $\CAT$-space if, and only if,
\begin{equation}
\langle\vv{xy},\vv{uv}\rangle \leq d(x,y)d(u,v)\tag{CS}\label{cauchy}
\end{equation}
for all $x,y,u,v \in X$.

\subsection{(Quasi)convexity over Hadamard spaces} Let $(X,d)$ be a Hadamard space. A function $h: X \rightarrow \R$ is called \emph{convex} if
\begin{equation*}
h((1-\lambda)x \oplus \lambda y) \leq (1-\lambda) h(x) + \lambda h(y)
\end{equation*}
for all $x,y \in X$ and $\lambda \in [0,1]$. A particular example of a convex function that we will need later is the Busemann function $b_r:X\to\mathbb{R}$ associated with a geodesic ray $r$, defined by
\[
b_r(x):=\lim_{t\to\infty} (d(x,r(t))-t).
\]
In particular, $b_r$ is convex and nonexpansive, i.e.\ 1-Lipschitz (cf.\ \cite[Example 2.2.10]{Bacak2014}).

As motivated in the introduction, we will here be focused on the weaker notion of quasiconvexity, and its respective strengthening of strong quasiconvexity. In that vein, a function $h: X \rightarrow \R$ is called
\begin{itemize}
\item[(i)]\emph{quasiconvex} if 
\begin{equation*}
h((1-\lambda)x \oplus \lambda y) \leq \mathrm{max}\{h(x),h(y)\}
\end{equation*}
for all $x,y \in X$ and $\lambda \in [0,1]$,
\item[(ii)]\emph{strongly quasiconvex} with modulus $\gamma > 0$ if 
\begin{equation*}
h((1-\lambda)x \oplus \lambda y) \leq \mathrm{max}\{h(x),h(y)\}- \lambda(1-\lambda)\frac{\gamma}{2}d^2(x,y)
\end{equation*}
for all $x,y \in X$ and $\lambda \in [0,1]$.
\end{itemize}

Next to notions of quasiconvexity, we will also need the following notions of semicontinuity: A function $h:X \rightarrow \mathbb{R}$ is called \emph{lower semicontinuous} (lsc) at $x \in X$ if 
\begin{equation*}
h(x) \leq \liminf_{k \rightarrow \infty} h(x_k)
\end{equation*}
for any sequence $\{x_k\}_k \subseteq X$ with $x_k \rightarrow x$. If instead it holds that 
\begin{equation*}
h(x) \geq \limsup_{k \rightarrow \infty} h(x_k)
\end{equation*}
for any sequence $\{x_k\}_k \subseteq X$ with $x_k \rightarrow x$, $h$ is called \emph{upper semicontinuous} (usc) at $x \in X$.

\subsection{Tangent cones}\label{tangent}

We will later utilize the so-called tangent cones of a $\CAT$-space as developed in the work of Nikolaev \cite{Nikolaev1995} (see also \cite{BridsonHaefliger1999} for further exposition). Fix a $\CAT$-space $(X,d)$. For two nonconstant geodesics $\gamma,\eta$ issuing from a point $x\in X$, their Alexandrov angle $\angle_x(\gamma,\eta)$ is defined by
\[
\angle_x(\gamma,\eta):=\lim_{s,t\to 0^+}\overline{\angle}_x(\gamma(s),\eta(t)),
\]
where $\overline{\angle}_x(y,z)$ is the comparison angle defined through the comparison triangle $\overline{\Delta}(\overline{x},\overline{y},\overline{z})$ of the geodesic triangle $\Delta(x,y,z)\subseteq X$ (see \cite{BridsonHaefliger1999}). Using the fact that $\angle_x$ defines a pseudometric on the set of all nonconstant geodesics issuing from $x$, we define $\Sigma'_xX$ to be the set of all equivalence classes of such geodesics under the equivalence relation defined by $\angle_x(\gamma,\eta)= 0$. The completion $(\Sigma_xX,\angle_x)$ of the space $(\Sigma'_xX,\angle_x)$ is called the metric space of directions from $x$ and the tangent cone $T_xX$ of $X$ at $x$ is then the Euclidean cone over $\Sigma_xX$, that is $T_xX:=(\Sigma_xX\times [0,\infty) )/\sim$ where $(\gamma,t)\sim(\eta,s)$ if, and only if, $t=s=0$ or $t=s>0$ and $\gamma=\eta$. For brevity, we write $t\gamma$ for the equivalence class $[(\gamma,t)]_\sim$ and given $u=t\gamma$ and $\lambda\geq 0$, we write $\lambda u:=(\lambda t)\gamma$. On $T_xX$, we define a metric
\[
d_x(t\gamma,s\eta):=\sqrt{t^2+s^2-2ts\cos\angle_x(\gamma,\eta)}.
\]
By the results of Nikolaev \cite{Nikolaev1995}, it follows that $(T_xX,d_x)$ is a Hadamard space. We further write $0_x:=0\gamma$ as well as $\norm{t\gamma}_x:=d_x(0_x,t\gamma)=t$ and define
\[
g_x(t\gamma,s\eta):=\frac{1}{2}\left(\norm{t\gamma}_x^2+\norm{s\eta}_x^2-d_x^2(t\gamma,s\eta)\right)=ts\cos\angle_x(\gamma,\eta).
\]
It can be easily seen that $g_x(t\gamma,s\eta)=\svv{0_x}{t\gamma}{0_x}{s\eta}_x$, where $\langle\vv{\cdot},\vv{\cdot}\rangle_x$ is the quasi-inner product on $T_xX$, and so $g_x(t\gamma,s\eta)\leq \norm{t\gamma}_x\norm{s\eta}_x$ by \eqref{cauchy} on $T_xX$. Further, one has $g_x(t\gamma,t\gamma)=\norm{t\gamma}_x^2$, $g_x(t\gamma,s\eta)=g_x(s\eta,t\gamma)$ and $g_x(t\gamma,s\eta)=tg_x(\gamma,s\eta)$.

Similar as e.g.\ Ohta \cite{Ohta2012} we define the inverse exponential map in a general Hadamard space to be the function $\log_x:X\to T_xX$ given by $\log_x a:=d(x,a)\gamma_{x,a}$ for $a\neq x$, where $\gamma_{x,a}$ is the (unique) geodesic connecting $x$ to $a$, as well as $\log_x x:=0_x$. The central property of the inverse exponential map $\log_x$ in relation to the pseudo-inner product $g_x$ that we need is the following:

\begin{lemma}[{essentially \cite[Proposition 2.16]{ChaipunyaKohsakaKumam2021}}]\label{tangentCat0}
For any $x,a,b\in X$:
\[
g_x(t\log_x a,s\log_x b)\geq \frac{ts}{2}(d^2(x,a)+d^2(x,b)-d^2(a,b)).
\]
\end{lemma}

\subsection{Weak convergence in Hadamard spaces}

We now introduce the notion and basic properties of weak convergence in Hadamard spaces. First introduced by Jost \cite{Jost94}, in \cite{ESPINOLA2009}, Esp\'inola and Fern\'andez-Le\'on show that the same notion was obtained by Kirk and Panyanak in \cite{KIRK2008}, who brought Lim's $\Delta$-convergence \cite{Lim1976} into Hadamard spaces. We will thus denote weak convergence by $\Delta$-convergence, notationally agreeing with \cite{KHATIBZADEHMOHEBBI2021}, whose results we will crucially rely on in Section \ref{the_existence_issue} later on for establishing the existence of equilibrium points. For a thorough discussion of weak convergence in Hadamard spaces, see e.g.\ \cite{Bacak2014}.

Let $(X,d)$ be a Hadamard space and $\{x_k\}_k \subseteq X$ be a bounded sequence. For $x \in X$, set
\[
r(x, \{x_k\}_k) := \limsup_{n \rightarrow \infty} d(x,x_k).
\]
The \emph{asymptotic radius} $r(\{x_k\}_k)$ is given by 
\[
r(\{x_k\}_k) := \mathrm{inf} \{ r(x, \{x_k\}_k)\mid x \in X\},
\]
and the \emph{asymptotic center} $A(\{x_k\}_k)$ of $\{x_k\}_k$ is the set
\[
A(\{x_k\}_k) := \{x \in X \mid r(\{x_k\}_k) = r(x, \{x_k\}_k)\}.
\]
It is known that $A(\{x_k\}_k)$ is a singleton in Hadamard spaces (see e.g.\ \cite{KIRK2008}).

A bounded sequence $\{x_k\}_k \subseteq X$ is then said to $\Delta$-converge to $x \in X$ if $x$ is the unique asymptotic center of every subsequence of $\{x_k\}_k$. We write $x_k \rightarrow^\Delta x$ and call $x$ the $\Delta$-limit of $\{x_k\}_k$. We further call a point $x \in X$ a $\Delta$-cluster point of a sequence $\{x_k\}_k$ if there is a subsequence $\{x_{k_l}\}_l$ of $\{x_k\}_k$ such that $ x_{k_l} \rightarrow^\Delta x$.

We will later rely on various results regarding this notion of weak convergence, the most essential of which we list in the following:

\begin{lemma}[{cf.\ \cite[Proposition 3.5]{KIRK2008}}]\label{Lemma7P}
Any bounded sequence in $X$ has a $\Delta$-cluster point.
\end{lemma}

\begin{lemma}[{cf.\ \cite[Lemma 3.2.1]{Bacak2014}}]\label{Lemma8P}
Let $K\subseteq X$ be a closed and convex set and $\{x_k\}_k \subseteq K$. If the sequence $\{x_k\}_k$ $\Delta$-converges to a point $x \in X$, then $x \in K$.
\end{lemma}

Together, the above two lemmas in particular yield: 

\begin{lemma}[{cf.\ \cite[Proposition 3.6]{KIRK2008}}]\label{delta}
Every bounded closed convex set in $X$ is $\Delta$-compact.
\end{lemma}

We say that a function $h: X \rightarrow \R$ is \emph{$\mathit{\Delta}$-lower semicontinuous} ($\Delta$-lsc) at a point $x \in X$ if
\[
\liminf_{n \rightarrow \infty} h(x_k) \geq h(x)
\]
for each sequence $\{x_k\}_k$ with $x_k \rightarrow^\Delta x$. The property of $h: X \rightarrow \R$ being $\Delta$-upper semicontinuous ($\Delta$-usc) is defined analogously.

\begin{lemma}[{cf.\ \cite[Lemma 3.7]{Pischke2025}, extending \cite[Lemma 3.2.3]{Bacak2014}}]\label{Lemma9P}
Let $h: X \rightarrow \R$ be a quasiconvex lsc function. Then $h$ is $\Delta$-lsc.
\end{lemma}

\section{Equilibrium problems over Hadamard spaces and existence of solutions}\label{the_existence_issue}

We now fix the class of bifunctions and associated equilibrium problems studied in this paper: Given a Hadamard space $(X,d)$ and a bifunction $f: X \times X \rightarrow \R$, the task is to 
\begin{equation}
	\text{find } x^{*} \in X \text{ such that } f(x^{*},y) \geq 0 \text{ for all } y \in X.\tag{EP}\label{equil}
\end{equation}
We denote the associated solution set by $S(f)$. The focus will be on a special class of pseudomonotone, strongly quasiconvex and suitably Lipschitz bifunctions. As we throughout also utilize other auxiliary properties, we  enumerate the various assumptions featuring in this paper for a better overview (compare \cite{IusemLara2021}): 

\begin{enumerate}
\item [$(A1)$] For every $x \in X$, the function $f(x,\cdot)$ is lsc, and for every $y \in X$, the function $f(\cdot,y)$ is usc.\label{A1}
\item [$(A2)$] $f$ is pseudomonotone, i.e.
\[
f(x,y) \geq 0 \text{ implies } f(y,x) \leq 0
\]
for all $x,y \in X$.
\item [$(A3)$] $f$ is lsc (jointly in both arguments).
\item [$(A4)$] For every $x \in X$, the function $f(x,\cdot)$ is strongly quasiconvex with modulus $\gamma > 0$.
\item [$(A5)$] $f$ satisfies the following Lipschitz condition: there exists $\eta > 0$ such that
\[
f(x,z) - f(x,y) - f(y,z) \leq \eta \left(d^2(x,y) + d^2(y,z) \right)
\]
for all $x,y,z \in X$.
\item [$(A6)$] The Lipschitz constant $\eta$ and the modulus of strong quasiconvexity $\gamma$ are such that $12 \eta < \gamma$.
\end{enumerate}

While $(A1)$ and $(A2)$ are standard assumptions in the literature (see e.g.\ \cite{KHATIBZADEHMOHEBBI2021} for such assumptions in the context of equilibrium problems on Hadamard spaces), the assumptions $(A4)$, $(A5)$ and $(A6)$ are metric versions of the corresponding properties used in \cite{IusemLara2021} over $\mathbb{R}^d$, and replace the usual assumption of convexity of $f$. As in \cite{IusemLara2021}, we also note here that the standard assumption
\begin{enumerate}
\item [$(A0)$] $f(x,x) = 0$ for all $ x \in X$.
\end{enumerate}
is derivable from $(A2)$ and $(A5)$, although we will later also make use of it individually. $(A3)$ seems to be a more uncommon assumption, and we will in particular show that it can generally be avoided in the context of the convergence and existence theorems of the present paper, improving the results of \cite{IusemLara2021} (see also Remark \ref{firstAlgStrengthening} later on).

Before moving on to the question when and under what assumptions the above equilibrium problem \eqref{equil} has a solution in the nonlinear context of Hadamard spaces, we give a key example of a bifunction that illustrates the above assumptions $(A1)$ -- $(A6)$. Further examples of bifunctions satisfying the above conditions simultaneously can be found in Section 4 of \cite{IusemLara2021}.

\begin{example}[{essentially \cite[Corollary 3.1]{IusemLara2021}}]\label{ex_min_is_equi}	
Let $h: X \rightarrow \R$ be lsc and strongly quasiconvex with modulus $\gamma > 0$. Consider the bifunction $f_h: X \times X \rightarrow \R$ given by 
\[
f_h(x,y) := h(y)-h(x) \text{ for all }x,y \in X.
\]
It is then easy to see that $S(f_h)=\argmin_{x\in X} h(x)$, that is the solutions of the equilibrium problem for $f_h$ correspond exactly to the solutions of the minimization problem for $h$. Further, $f_h$ fulfills $(A1)$, $(A2)$ and $(A4)$ -- $(A6)$. Here, $(A1)$, $(A2)$ and $(A4)$ are rather immediate (see also the proof of Corollary \ref{minExists} later on) and $(A5)$ as well as $(A6)$ follow from the fact that  $f_h(x,z) -f_h(x,y) -f_h(y,z) = 0$.
\end{example}

Now, at first note that the strong quasiconvexity and pseudomonotonicity of $f$ imply that $S(f)$ is either empty or a singleton, which generalizes parts of \cite[Proposition 3.1]{IusemLara2021} to our nonlinear setting (and in particular extends \cite[Lemma 3.1]{Pischke2025} from minimization problems to equilibrium problems).

\begin{lemma}\label{uniqueness}
Suppose that $f$ satisfies $(A0)$, $(A2)$ and $(A4)$. Then $S(f)$ is at most a singleton.
\end{lemma}
\begin{proof}
If $x_1, x_2 \in S(f)$ are such that $x_1 \neq x_2$, then $f(x_1,x_2) \geq 0$ and $f(x_2,x_1) \geq 0$. By pseudomonotonicity, we have $f(x_2,x_1) = 0$. Using the strong quasiconvexity of $f(x_2, \cdot)$, we get
\[
f\left(x_2,\tfrac{1}{2} x_2 \oplus \tfrac{1}{2}x_1\right) \leq \mathrm{max}\{f(x_2,x_1),f(x_2,x_2)\} - \frac{\gamma}{8}d^2(x_1,x_2).
\]
Thus, with $(A0)$, we have $f\left(x_2,\tfrac{1}{2} x_2 \oplus \tfrac{1}{2}x_1\right) < 0$, contradicting that $x_2$ is an equilibrium point.
\end{proof}

In the context of $\R^d$, Iusem and Lara in \cite{IusemLara2021} further prove that there always exists an equilibrium point whenever the assumptions $(A0)$, $(A1)$, $(A2)$ and $(A4)$ hold (cf.\ \cite[Proposition 3.1]{IusemLara2021}). The proof given in \cite{IusemLara2021} relies on local compactness arguments which are not available in general Hadamard spaces (and already not in infinite-dimensional Hilbert spaces). However, by generalizing a weak compactness argument developed in \cite{KHATIBZADEHMOHEBBI2021} for the case where $f$ is convex in its right argument (extending the approach taken in \cite{IusemKassaySosa09,IUSEM2003} over Hilbert spaces and in \cite{COLAO201261} over Hadamard manifolds), we can generalize all of \cite[Proposition 3.1]{IusemLara2021} to Hadamard spaces.

For this, we rely on two notions of coercivity. The first is the following usual notion for functions of one argument:

\begin{definition} 
Let $o \in X$. A function $h: X \rightarrow \R$ is called \emph{1-supercoercive} if 
\[
\liminf_{d(o,x)\rightarrow\infty}\frac{h(x)}{d(o,x)}>0.
\]
\end{definition}

The second is a notion for bifunctions, originally introduced in a linear setting in \cite{IusemKassaySosa09}. 

\begin{definition}[{cf.\ \cite[Assumption P5]{KHATIBZADEHMOHEBBI2021}}]
Let $o \in X$. A bifunction $f: X \times X \rightarrow \R$ is called \emph{coercive as a bifunction} if for any sequence $\{x_n\}_n \subseteq X$ satisfying $\lim_{n\to\infty} d(o,x_n) = \infty$, there exists $u \in X$ and $n_0 \in \N$ such that $f(x_n, u) \leq 0$ for all $n \geq n_0$.
\end{definition} 

Generalizing parts of \cite[Theorem 1]{Lara2022} to the Hadamard setting, we now first show that any lsc, strongly quasiconvex function is 1-supercoercive.

\begin{lemma}\label{shortt}
Let $h: X \rightarrow \R$ be lsc and strongly quasiconvex with modulus $\gamma$. Then $h$ is 1-supercoercive. 
\end{lemma}
\begin{proof}
Fix $o,y\in X$ and let $\{x_k\}_k \subseteq X$ with $d(o,x_k) \rightarrow \infty$. Define
\[
z_k := \frac{d(o,x_k)}{1+d(o,x_k)}y \oplus \frac{1}{1+d(o,x_k)}x_k.
\]
Notice that $\{z_k\}_k$ is bounded as $d$ is convex. As $h$ is strongly quasiconvex, we deduce (dividing by $d(o,x_k)$) that
\begin{equation*}
\frac{h(z_k)}{d(o, x_k)} \leq \max\left\{\frac{h(y)}{d(o,x_k)}, \frac{h(x_k)}{d(o,x_k)}\right\} - \frac{\gamma}{2} \frac{d^2(y,x_k)}{\left(1+d(o,x_k)\right)^2}.
\end{equation*}
We will now show that $\liminf_{k\to\infty} h(z_k) > -\infty$. Suppose not. Then, there exists a subsequence $\{z_{k_l}\}_l$ of $\{z_k\}_k$ such that $h(z_{k_l}) \rightarrow -\infty$. However, as $\{z_k\}_k$ and thus $\{z_{k_l}\}_l$ are bounded, by Lemma \ref{Lemma7P}, there exists a subsequence $\{z_{k_{l_j}}\}_j$ of $\{z_{k_l}\}_l$ with $z_{k_{l_j}} \rightarrow^\Delta z$ for some $z \in X$. As further with Lemma \ref{Lemma9P}, $h$ is $\Delta$-lsc, we obtain that $\liminf_{j\to\infty} h(z_{k_{l_j}}) \geq h(z)$ contradicting $h(z_{k_l}) \rightarrow -\infty$. Hence, $\liminf_{k\to\infty} h(z_k) > -\infty$ after all and thus $\liminf_{k\to\infty} h(z_k)/d(o,x_k) \geq 0$. Further, $h(y)/d(o,x_k) \to 0$. As however $\lim_{k\to\infty} \frac{\gamma}{2} d^2(y,x_k)/\left(1+d(o,x_k)\right)^2 = \frac{\gamma}{2}$, it must hold that $\liminf_{k\to\infty} h(x_k)/d(o,x_k) \geq \frac{\gamma}{2} > 0$. Hence $h$ is 1-supercoercive.
\end{proof}

From this it follows that, under suitable conditions, $f$ is coercive as a bifunction:

\begin{lemma}\label{coercive}
Let $f: X \times X \rightarrow \R$ be a bifunction satisfying $(A1)$, $(A2)$ and $(A4)$. Then $f$ is coercive as a bifunction. 
\end{lemma}
\begin{proof}
Fix $x \in X$. From Lemma \ref{shortt}, we get that $\liminf_{d(o,y) \to \infty} f(x,y)/d(o,y) >0$, using also $(A1)$ and $(A4)$. This in particular implies that for any sequence $\{y_n\}_n\subseteq X $ with $\lim_{n \rightarrow \infty}d(o,y_n) = \infty$, we find an $n_0$ such that $f(x,y_n) \geq 0 $ for all $n \geq n_0$. Using $(A2)$, that is pseudomonotonicity of $f$, implies the result.
\end{proof}

Next to the above coercivity properties, the proof that $S(f)$ is nonempty under suitable assumptions on $f$ rests on two lemmas derived from \cite{KHATIBZADEHMOHEBBI2021}. The first is a nonlinear version of the so-called KKM Lemma \cite{KnasterKuratowskiMazurkiewicz1929} (see also \cite[Lemma 1]{Fan1961}), derived for finite-dimensional Hadamard manifolds in \cite{COLAO201261} and then for Hadamard spaces in \cite{KHATIBZADEHMOHEBBI2021}.

\begin{lemma}[{cf.\ \cite[Lemma 1.8]{KHATIBZADEHMOHEBBI2021}}]\label{Lemma1.8}
Let $G:X\rightarrow 2^X$ be such that for each $x \in X$, $G(x)$ is $\Delta$-closed. Suppose that
\begin{enumerate}
\item[(a)] for all $x_1,\dots,x_m \in X$, $\mathrm{conv}(\{x_1,\dots,x_m\})\subseteq \bigcup_{i=1}^m G(x_i)$,
\item[(b)] there exists $x_0 \in X$ such that $G(x_0)$ is $\Delta$-compact.
\end{enumerate}
Then $\bigcap_{x\in X}G(x)\neq \emptyset$.
\end{lemma}

The second lemma states that if it is possible to find an equilibrium point $x$ for $f$ on some closed ball in $X$ and one point $y$ in the interior of that ball such that $f(x,y) \leq 0$, then $x$ is an equilibrium point for $f$ on the whole space $X$. To formulate it, we consider the following definitions akin to \cite{KHATIBZADEHMOHEBBI2021}: Fix $o\in X$. For any $n \in \N$, define
\[
X_n := \{ x \in X \mid d(o,x) \leq n\}\text{ with }X_n^\circ := \{ x \in X \mid d(o,x) < n\}.
\]
Further, given a bifunction $f: X \times X \rightarrow \R$ and $y\in X$ as well as $n \in \N$, define
\[
L_f(n,y):= \{ x \in X_n \mid f(y,x) \leq 0\}.
\]

We can now adapt \cite[Lemma 2.1]{KHATIBZADEHMOHEBBI2021} to our setting of strongly quasiconvex bifunctions. Here, we follow a modified version of the proof given for a preceding result of this type in linear spaces by Iusem and Sosa (cf.\ \cite[Lemma 3.7]{IUSEM2003}).

\begin{lemma}\label{ball}
Let $f$ be a bifunction satisfying $(A0)$, $(A2)$ and $(A4)$. If for some $n \in \N$ and some $\overline{x} \in \bigcap_{y\in X_n}L_f(n,y)$ there exists $\overline{y} \in X^\circ_n$ such that $f(\overline{x},\overline{y})\leq 0$, then $f(\overline{x},y)\geq 0$ for all $y \in X$.
\end{lemma}
\begin{proof}
Notice that by definition 
\[
\bigcap_{y\in X_n}L_f(n,y) = \bigcap_{y\in X_n} \{ x \in X_n \mid f(y,x) \leq 0\} = \{ x \in X_n \mid f(y,x) \leq 0 \text{ for all } y \in X_n\}
\]
for any $n \in \N$. Thus, if $\overline{x} \in \bigcap_{y\in X_n}L_f(n,y)$ for some $n \in \N$, we know $f(y,\overline{x}) \leq 0$ for any $y \in X_n$. We now show that $f(\overline{x},y) \geq 0$ for all $y \in X_n$. For that, let $y\in X_n$. If $y=\overline{x}$, the claim is obvious by $(A0)$. So suppose that $y\neq\overline{x}$ and define $w_\lambda:= (1-\lambda)\overline{x}\oplus \lambda y$ for $\lambda\in [0,1]$. By $(A0)$ and $(A4)$, we get
\[
0=f(w_\lambda,w_\lambda)\leq \max\{f(w_\lambda,y),f(w_\lambda,\overline{x})\}-\lambda(1-\lambda) \frac{\gamma}{2}d^2(y,\overline{x}).
\]
As $w_\lambda\in X_n$ since balls are convex, we get $f(w_\lambda,\overline{x})\leq 0$ from the above. Since we have $\lambda(1-\lambda) \frac{\gamma}{2}d^2(y,\overline{x})>0$ for $\lambda\in (0,1)$, it must hold that $f(w_\lambda,y)> 0$ for all such $\lambda$. By the continuity of geodesics, we have $w_\lambda\to\overline{x}$ for $\lambda\to 0$ and so we get
\[
f(\overline{x},y) \geq \limsup_{\lambda\to 0} f(w_\lambda,y)\geq 0
\]
since $f$ is usc in its left argument. It is left to show that $f(\overline{x},y) \geq 0$ for all $y \in X\setminus X_n$. For that, let $y \in X \backslash X_n$ be arbitrary. Then, with $\overline{y} \in X^\circ_n$ and by the continuity of geodesics, there exists $\lambda \in (0,1)$ such that $z := (1-\lambda) y \oplus\lambda \overline{y} \in X_n$. As $z\in X_n$ and by $(A4)$, that is strong quasiconvexity, we get
\[
0 \leq f(\overline{x},z) = f(\overline{x}, (1-\lambda) y \oplus\lambda \overline{y}) \leq \mathrm{max}\{f(\overline{x},y),f(\overline{x},\overline{y})\}-\lambda(1-\lambda) \frac{\gamma}{2}d^2(y,\overline{y}).
\]
As by assumption $f(\overline{x},\overline{y}) \leq 0$ and as further $\lambda(1-\lambda)\frac{\gamma}{2}d^2(y,\overline{y})>0$, it must hold that $f(\overline{x},y)>0$ and we are done.
\end{proof}

We can now easily transfer \cite[Theorem 2.4]{KHATIBZADEHMOHEBBI2021} to the strongly quasiconvex case, and thereby extend \cite[Proposition 3.1]{IusemLara2021} to the setting of Hadamard spaces.

\begin{theorem}\label{nonemptyness_thrm}
Let $f$ be a bifunction satisfying $(A0)$, $(A1)$, $(A2)$ and $(A4)$. Then $S(f)$ is nonempty.
\end{theorem}
\begin{proof}
By \cite[Lemma 2.3]{KHATIBZADEHMOHEBBI2021}, $f$ is properly quasimonotone, that is $\min_{x\in A}f(x,y)\leq 0$ for every finite set $A \subseteq X$ and every $y \in \mathrm{conv}(A)$. Further, by Lemma \ref{coercive}, $f$ is coercive as a bifunction. The result now follows by essentially exactly the same arguments as in \cite[Theorem 2.4]{KHATIBZADEHMOHEBBI2021} but, for the sake of completeness, we present them here. Let $n \in \N$. We use Lemma \ref{Lemma1.8} with $X_n$ instead of $X$ and $G(y) := L_f(n,y)$. To show that we are in its setting, consider $x_0,\dots, x_k \in X_n$ and $\overline{x} \in \mathrm{conv}(\{x_0,\dots,x_k\})$. But then $\overline{x} \in \bigcup_{i=0}^kL_f(n,x_i) $ as $\overline{x} \in X_n$ by convexity of $X_n$ and $\mathrm{min}_{0\leq i \leq k}f(x_i, \overline{x}) \leq 0$ by proper quasimonotonicity of $f$. Further, we have that $G(y)=\{x\in X_n\mid f(y,x) \leq 0\}$ is closed as $f$ is lsc and convex as $f$ is (strongly) quasiconvex in its second argument. Therefore, it is $\Delta$-closed by Lemma \ref{Lemma8P}. $G(y)$ is also bounded as it is contained in $X_n$. Thus, by Lemma \ref{delta}, it is $\Delta$-compact for all $y \in X$. We can hence apply Lemma \ref{Lemma1.8}, deducing that $\bigcap_{y\in X_n} L_f(n,y) \neq \emptyset$ for all $n \in \N$. For each $n$, we choose $x_n \in \bigcap_{y \in X_n} L_f(n,y)$ and distinguish two cases.

\textbf{Case 1:} There is an $n \in \N$ such that $d(o,x_n) < n$. Then, $x_n \in S(f)$ by Lemma \ref{ball}, and so we are done.

\textbf{Case 2:} $d(o,x_n) = n $ for all $n \in \N$. But then, as $f$ is coercive as a bifunction by Lemma \ref{coercive}, there exists $u \in X$ and $n_0\in\mathbb{N}$ such that $f(x_n,u) \leq 0$ for all $n \geq n_0$. Consider $n' \geq n_0$ such that $d(o,u) < n'$. Then $f(x_{n'},u) \leq 0$ and $u \in X_{n'}^\circ$. Again we are done by Lemma \ref{ball}.
\end{proof}

Using the above Theorem \ref{nonemptyness_thrm}, we can in particular show that a lsc strongly quasiconvex function $h: X \rightarrow \R$ over a Hadamard space has a minimizer, answering a question left open in the work \cite{Pischke2025} of the second author.

\begin{corollary}\label{minExists}
Let $h : X \rightarrow \R$ be a lsc strongly quasiconvex function. Then  $\mathrm{argmin}_{x\in X}h(x) $ is nonempty. 
\end{corollary}
\begin{proof}
As in Example \ref{ex_min_is_equi}, we define $f_h$ by $f_h(x,y) := h(y) - h(x)$. This function clearly satisfies $(A0)$ and $(A1)$. Further, if $f_h(x,y) \geq 0$, then $h(y) - h(x) \geq 0$ and so $h(x) - h(y) \leq 0$, that is $f_h(y,x) \leq 0$. Hence $f_h$ is pseudomonotone and so satisfies $(A2)$. As $h$ is strongly quasiconvex and remains strongly quasiconvex under addition of a constant, $f_h(x,\cdot)$ is strongly quasiconvex for every $x \in X$, and so satisfies $(A4)$. We can thus apply Theorem \ref{nonemptyness_thrm} and obtain $x^{*} \in X$ such that $f_h(x^{*}, y) \geq 0$ for all $y \in X$. This however immediately yields that $x^{*} \in \mathrm{argmin}_{x\in X}h(x)$.
\end{proof}

\section{A proximal point algorithm}

\subsection{The method and preliminaries}

We now move to the first method for solving the equilibrium problem \eqref{equil} over a Hadamard space $(X,d)$ for a given bifunction $f: X\times X \rightarrow \R$ that satisfies assumptions among $(A1)$ -- $(A6)$. As before, we write $S(f)$ for the associated solution set.

The method we consider is the following proximal point type method: For a given start point $x_0\in X$ and a sequence $\{\beta_k\}_{k} \subseteq (0,\infty)$, recursively define $x_{k+1}$ by choosing
\[
x_{k+1} \in \argmin_{y\in X}\left\{f(x_k,y)+ \frac{1}{2\beta_k}d^2(x_k,y)\right\}.\tag{PPA}\label{PPA}
\]
This method represents a metric variant of one of the proximal point type methods studied by Iusem and Lara (cf.\ \cite[Algorithm 1]{IusemLara2021}) in the context of the context of the above equilibrium problem \eqref{equil} over $\mathbb{R}^d$, corresponding to the second type of method discussed in the introduction which regularizes the bifunction only in its second argument. Concretely, as also highlighted throughout the introduction, the work \cite{IusemLara2021} considers the problem over a closed and convex subset $K \subseteq \mathbb{R}^d$, and for a corresponding bifunction $f: K\times K \rightarrow \R$ satisfying $(A1)$ -- $(A6)$, suitably formulated over $K$ with a corresponding solution set $S(K,f)$. We however omit such restrictions here as any closed and convex subset of a Hadamard space (and so in particular of a Hilbert space) is naturally a Hadamard space again, so that considering the problem \eqref{equil} over $X$ already offers the full generality here. This also applies to the existence results established previously.

In the setting described above, Iusem and Lara show in \cite{IusemLara2021} that this proximal point type algorithm is well-defined and that it converges under suitable assumptions on the parameters:

\begin{theorem}[{cf.\ \cite[Theorem 3.1]{IusemLara2021}\label{ILconv}}]
Let $X = \R^d$ and $K \subseteq X$ be closed and convex. Let $f:K\times K\to\mathbb{R}$ be a bifunction satisfying $(A1)$ -- $(A6)$ and let $\{\beta_k\}_k \subseteq (\frac{1}{\gamma - 8\eta}, \frac{1}{4\eta})$. Let $\{x_k\}_k$ be the sequence generated by \eqref{PPA} over $K$. Then $\{x_k\}_k$ converges to a point $x^{*} \in S(K,f)$.
\end{theorem}	

Notice that under $(A6)$, it holds that $\lowerbound < \upperbound$ so that the interval $(\frac{1}{\gamma - 8\eta}, \frac{1}{4\eta})$ is non-empty and the parameter restriction $\{\beta_k\}_k \subseteq (\frac{1}{\gamma - 8\eta}, \frac{1}{4\eta})$ is satisfiable.\footnote{A slight weakening of assumption $(A6)$ is discussed in \cite[Remark 87]{GLM2025} (in the context of a variant of the above method including inertia terms and over-relaxations as developed in \cite{GLM2023,GradLaraMarca24}), with a corresponding refined parameter restriction on $\{\beta_k\}_k$. While we deem it quite likely that our present results can be extended to this weakened assumption, we here focus on $(A6)$ for simplicity.} In particular, as highlighted in the introduction, we will be able to weaken this restriction to $\{\beta_k\}_k\subseteq [\lowerbound,\upperbound]$ later on.

Given any function $h: X \rightarrow \R$, we define its proximal map $\mathrm{Prox}_{\beta h}$ by 
\begin{equation*}
\mathrm{Prox}_{\beta h}(x) := \argmin_{y\in X}\left\{h(y) + \frac{1}{2\beta}d^2(x,y)\right\}
\end{equation*}
for $\beta>0$ and $x \in X$. Using that definition, as already highlighted in the introduction, the method \eqref{PPA} assumes a much more compact form with
\[
x_{k+1} \in \mathrm{Prox}_{\beta_k f(x_k, \cdot)}(x_k).
\]

We now first show the well-definedness of the method \eqref{PPA} above, i.e.\ that $\mathrm{Prox}_{\beta_k f(x_k, \cdot)}(x_k)\neq\emptyset$ for any $k\in\mathbb{N}$, provided that $\beta_k \leq \upperbound$. The first crucial result in this direction is the following:

\begin{lemma}[{cf.\ \cite[Lemma 3.4]{Pischke2025}}] \label{Lemma6P}
If $h: X \rightarrow \R$ is bounded from below and 1-supercoercive, then any sequence $\{y_k\}_k$ with $h(y_k) \rightarrow \inf_{x\in X}h(x)$ is bounded.
\end{lemma}

The above result can be employed to yield the following:

\begin{lemma}[{cf.\ \cite[Lemma 3.8]{Pischke2025}}]\label{Lemma10P}
Let $h: X \rightarrow \R$ be a quasiconvex, lsc function which is bounded from below. Then, for any $x \in X$ and $\beta>0$, $\mathrm{Prox}_{\beta h}(x)$ is nonempty.
\end{lemma}

This can now be extended to yield the following, guaranteeing an analogous result to Lemma \ref{Lemma10P} above for suitable bifunctions and parameters, provided $S(f)$ is nonempty.

\begin{proposition}\label{Lemma10}
Let $f: X \times X \rightarrow \R$ be a function which is quasiconvex and lsc in its right argument and further satisfies $\textrm{$(A5)$}$. Assume that there exist $\overline{x} \in X$ and $t \in \R$ with $f(\overline{x},y) \geq t$ for all $y \in X$. Then, for any $ z \in X$ and $\beta \leq \upperbound$, $\mathrm{Prox}_{\beta f(z,\cdot)}(z)$ is nonempty. 
\end{proposition}
\begin{proof}
Fix $z \in X$ and let $\overline{x}$, $t$ be as above. Then it follows from $(A5)$ that
\begin{align*}
f(\overline{x}, y) - f(\overline{x}, z) - \eta(d^2(\overline{x},z)+d^2(z,y)) \leq f(z,y) \text{ for all } y \in X.
\end{align*}
As $f(\overline{x}, y) \geq t$, we can set $G := -t + f(\overline{x},z)+\eta d^2(\overline{x}, z)$, obtaining
\begin{equation}
-G -\eta d^2(z,y) \leq f(z,y) \text{ for all } y \in X.\tag{$+$}\label{keinBock}
\end{equation}
We define $h_z (y) := f(z,y) + \frac{1}{2\beta}d^2(z,y)$ and show that $h_z$ is bounded from below. For this, notice that (\ref{keinBock}) implies
\begin{align*}
h_z(y) = f(z,y) + \frac{1}{2\beta}d^2(z,y) \geq \left(\frac{1}{2\beta}-\eta\right)d^2(z,y) - G.		
\end{align*}
But now as $\beta \leq \upperbound$, we obtain that $\frac{1}{2\beta} - \eta \geq 2 \eta - \eta = \eta > 0$. Thus we have $h_z(y) \geq \eta \ d^2(z,y) - G \geq -G$, and so $h_z$ is bounded below. Take $\{y_k\}_k$ such that $h_z(y_k) \rightarrow \inf_{x\in X} h_z(x)$. As further 
\begin{align*}
\liminf_{d(z,y) \rightarrow \infty} \frac{h_z(y)}{d(z,y)}\geq \liminf_{d(z,y) \rightarrow \infty} \left(\frac{\eta d^2(z,y)}{d(z,y)} - \frac{G}{d(z,y)}\right) \geq \liminf_{d(z,y) \rightarrow \infty}\frac{\eta d^2(z,y)}{d(z,y)} > 0,
\end{align*}
Lemma \ref{Lemma6P} yields that $\{y_k\}_k$ is bounded. By Lemma \ref{Lemma7P}, it has a $\Delta$-cluster point, i.e.\ there exists a subsequence $\{y_{k_l}\}_l$ of $\{y_k\}_k$ with $y_{k_l} \rightarrow^\Delta y$ . As $f(z,\cdot)$ is quasiconvex and lsc, by Lemma \ref{Lemma9P}, $f(z,\cdot)$ is $\Delta$-lsc. Further, as \eqref{CN} implies that $y \mapsto\frac{1}{2\beta}d^2(z,y)$ is convex and as it is also lsc, Lemma \ref{Lemma9P} also yields that it is $\Delta$-lsc. Hence, $h_z$ is $\Delta$-lsc and 
\begin{align*}
h_z(y) \leq \liminf_{l \rightarrow \infty} h_z(y_{k_l}) = \inf_{x\in X} h_z(x).
\end{align*} 
Thus, $y$ is a minimizer of $h_z$ and hence $y \in \mathrm{Prox}_{\beta f(z,\cdot)}(z)$.
\end{proof}

With Proposition \ref{Lemma10}, it follows in particular that whenever $(A1)$, $(A4)$ and $(A5)$ hold and whenever $S(f) \neq \emptyset$ and $\beta_k  \leq \upperbound$ for all $k \in \N$, then also $\mathrm{Prox}_{\beta_k f(x_k, \cdot)}(x_k)\neq\emptyset$ for any $k \in \N$. In particular, \eqref{PPA} is well-defined under these conditions and the parameter restriction $\{\beta_k\}_k\subseteq [\lowerbound,\upperbound]$.

The most essential result for strongly quasiconvex functions and their proximal maps is the following inequality, which is an extension of \cite[Proposition 7]{Lara2022} to the metric context of Hadamard spaces.

\begin{lemma}[{cf.\ \cite[Lemma 3.9]{Pischke2025}}]\label{Lemma11P}
Let $h: X \rightarrow \R$ be a strongly quasiconvex function with modulus $\gamma >0$ and let $\beta > 0$ and $x\in X$. If $\overline{x}\in \mathrm{Prox}_{\beta h}(x)$, then
\[
h(\overline{x}) \leq \max\{h(y),h(\overline{x})\} + \frac{\lambda}{2}\Big(\frac{\lambda}{\beta}-\gamma + \lambda \gamma \Big )d^2(y,\overline{x}) + \frac{\lambda}{\beta}\langle \vv{x\overline{x}}, \vv{\overline{x}y} \rangle
\] 
for all $y \in X$ and $\lambda \in [0,1]$.
\end{lemma}

Akin to \cite{IusemLara2021}, we can also show that fixed points of this mixed proximal point map are exactly the solutions of the equilibrium problem. 

\begin{proposition}\label{fixpoint_problem}
Let $f: X\times X \rightarrow \R$ be a bifunction satisfying $(A0)$ and $(A4)$. Then
\[
S(f) = \mathrm{Fix}\left(\argmin_{y\in X}\left\{f(\cdot,y) + \frac{1}{2\beta}d^2(\cdot,y)\right\}\right)
\]
for any $\beta > 0$.
\end{proposition}
\begin{proof}
Let $\overline{x}$ be a fixed point of $\mathrm{argmin}_{y\in X}\{f(\cdot,y) + \frac{1}{2\beta}d^2(\cdot,y)\}$, i.e.\ let $\overline{x} \in \mathrm{Prox}_{\beta f(\overline{x},\cdot)}(\overline{x})$. By $(A4)$, we get that $x \mapsto f(\overline{x},x)$ is strongly quasiconvex with modulus $\gamma$. Thus, invoking \cite[Proposition 3.10]{Pischke2025}, we derive that $\overline{x} \in \mathrm{argmin}_{y \in X}f(\overline{x},y)$. But then $\overline{x}$ is an equilibrium point for $f$, as $f(\overline{x},y) \geq f(\overline{x},\overline{x})$ for all $y \in X$ and  $f(\overline{x},\overline{x}) = 0$ by $(A0)$.

For the other direction, suppose that $\overline{x} \in S(f)$. Then $\overline{x} \in \mathrm{argmin}_{y \in X}f(\overline{x},y)$ as $f(\overline{x},y) \geq 0$ for all $y \in X$ and $f(\overline{x},\overline{x}) = 0$ by $(A0)$. But then again, by \cite[Proposition 3.10]{Pischke2025}, we get that $\overline{x} \in \mathrm{Prox}_{\beta f(\overline{x}, \cdot)}(\overline{x})$.
\end{proof}

In particular, also similar to \cite{IusemLara2021}, we want to remark that this provides a convenient stopping criterion as whenever $x_{k+1} = x_k$, one already knows that $\{x_k\} = S(f)$. 

The remainder of this section will now be concerned with establishing a variant of Theorem \ref{ILconv} in the context of Hadamard spaces, and in particular to endow it with effective information.

\subsection{Quantitative convergence results}\label{subsubsection_4_closed}
 
We begin by showing a result that will yield Fej\'er monotonicity of $\{x_k\}_k$ with respect to $S(f)$. For this, we adapt \cite[Proposition 3.4]{IusemLara2021} and its proof to the nonlinear setting. In particular, compared to this result, we highlight the more general inequalities \ref{fejer_easy} and \ref{fejer_less_easy_my_way} below. Only their instantiations with $\lambda = 1/2$ are considered explicitly in \cite[Proposition 3.4]{IusemLara2021}, but we will later require varying $\lambda$'s.

\begin{lemma}\label{fejer}
Let $f$ be a bifunction satisfying $(A2)$, $(A4)$ and $(A5)$. Let $\{\beta_k\}_k\subseteq (0,\infty)$. Suppose that \eqref{PPA} is well-defined and let $\{x_k\}_k$ be its generated sequence. Let $x^{*} \in S(f)$. Then, for every $k \in \N$, at least one of the following inequalities holds:
\begin{equation}
\frac{\lambda(1-\lambda)(1+\gamma\beta_k)}{2\beta_k}d^2(x_{k+1},x^{*}) \leq  \frac{\lambda}{2\beta_k}d^2(x_k,x^*) - \frac{\lambda}{2\beta_k}d^2(x_{k+1},x_k)\tag*{$(*)_1$}\label{fejer_easy}
\end{equation}
for all $\lambda \in [0,1]$, or 
\begin{equation}
\Big(\frac{\lambda(1-\lambda)(1+\gamma\beta_k)}{2\beta_k}-\eta\Big)d^2(x_{k+1},x^*) \leq \frac{\lambda}{2\beta_k}d^2(x_k,x^*) - \Big(\frac{\lambda}{2\beta_k}-\eta\Big)d^2(x_{k+1},x_k)\tag*{$(*)_2$}\label{fejer_less_easy_my_way}
\end{equation}
for all $\lambda \in [0,1]$.
\end{lemma}
\begin{proof}
Fix $k\in \N$. Using that $x_{k+1} \in \mathrm{Prox}_{\beta_k f(x_k,\cdot)}(x_k)$, applying Lemma \ref{Lemma11P} to the function $f(x_k, \cdot)$ with $y = x^{*}$ yields
\begin{align*}
&f(x_k, x_{k+1}) - \mathrm{max}\{f(x_k, x_{k+1}), f(x_k,x^*)\}\\
&\qquad\leq \frac{\lambda}{2}\Big(\frac{\lambda}{\beta_k} - \gamma + \lambda\gamma\Big)d^2(x_{k+1},x^*) + \frac{\lambda}{\beta_k} \Big\langle \vv{x_kx_{k+1}},\vv{x_{k+1}x^*}\Big\rangle
\end{align*}
for all $\lambda \in [0,1]$. We consider two cases:

\textbf{Case 1:} $f(x_k,x_{k+1}) \geq f(x_k, x^*)$. Then we have
\[
0 \leq \frac{\lambda}{2}\left(\frac{\lambda}{\beta_k} - \gamma + \lambda\gamma\right)d^2(x_{k+1},x^*) + \frac{\lambda}{\beta_k}\svv{x_k}{x_{k+1}}{x_{k+1}}{x^*}
\]
for all $\lambda \in [0,1]$. Using the definition of the quasi-inner product, we obtain
\[
0 \leq  \frac{\lambda}{2}\left(\frac{\lambda}{\beta_k} - \gamma + \lambda\gamma\right)d^2(x_{k+1},x^*) + \frac{\lambda}{2\beta_k} \left(d^2(x_k, x^{*})-d^2(x_{k+1},x_k)-d^2(x_{k+1},x^{*})\right)
\]
and thus 
\[
\frac{\lambda(1-\lambda)(1+\gamma\beta_k)}{2\beta_k}d^2(x_{k+1},x^{*}) \leq  \frac{\lambda}{2\beta_k}d^2(x_k,x^*) - \frac{\lambda}{2\beta_k}d^2(x_{k+1},x_k)
\]
for all $\lambda \in [0,1].$ This is \ref{fejer_easy}.

\textbf{Case 2:} $f(x_k,x_{k+1}) < f(x_k, x^*)$. Then
\begin{align*}
0 &\leq  \frac{\lambda}{2}\left(\frac{\lambda}{\beta_k} - \gamma + \lambda\gamma\right)d^2(x_{k+1},x^*) + \frac{\lambda}{\beta_k}\svv{x_k}{x_{k+1}}{x_{k+1}}{x^*} + f(x_k,x^*)- f(x_k, x_{k+1}) \\
&\leq  \frac{\lambda}{2}\left(\frac{\lambda}{\beta_k} - \gamma + \lambda\gamma\right)d^2(x_{k+1},x^*) + \frac{\lambda}{\beta_k}\svv{x_k}{x_{k+1}}{x_{k+1}}{x^*} + f(x_{k+1},x^*) \\
&\hphantom{\leq \mbox{}}+ \eta\Big(d^2(x_{k+1},x_k) + d^2(x_{k+1},x^*)\Big)
\end{align*}
for all $\lambda \in [0,1]$, where the second inequality follows from the fact that we get $f(x_k,x^{*})-f(x_k,x_{k+1})-f(x_{k+1},x^{*}) \leq \eta\big(d^2(x_{k+1},x_k) + d^2(x_{k+1},x^*)\big) $ with $(A5)$. Further, as $x^*$ is an equilibrium point and $f$ is  pseudomonotone, we get $f(x_{k+1},x^*) \leq 0$ and thus
\begin{align*}
0 &\leq  \frac{\lambda}{2}\left(\frac{\lambda}{\beta_k} - \gamma + \lambda\gamma\right)d^2(x_{k+1},x^*) + \frac{\lambda}{\beta_k}\svv{x_k}{x_{k+1}}{x_{k+1}}{x^*} \\
&\hphantom{\leq\mbox{}}+ \eta\Big(d^2(x_{k+1},x_k) + d^2(x_{k+1},x^*)\Big)
\end{align*}
for all $\lambda \in [0,1]$. Using the definition of the quasi-inner product, similar to Case 1, we obtain
\[
\Big(\frac{\lambda(1-\lambda)(1+\gamma\beta_k)}{2\beta_k}-\eta\Big)d^2(x_{k+1},x^*) \leq \frac{\lambda}{2\beta_k}d^2(x_k,x^*) - \Big(\frac{\lambda}{2\beta_k}-\eta\Big)d^2(x_{k+1},x_k)
\]
for all $\lambda \in [0,1]$. This is \ref{fejer_less_easy_my_way}.
\end{proof}

Using Lemma \ref{fejer}, we obtain the following:

\begin{lemma}\label{fejer_uni}
Let $f$ be a bifunction satisfying $(A2)$, $(A4)$ and $(A5)$. Let $\{\beta_k\}_k\subseteq (0,\infty)$. Suppose that \eqref{PPA} is well-defined and let $\{x_k\}_k$ be its generated sequence. Let $x^{*} \in S(f)$. Then, for every $k \in \N$ and $\lambda \in [0,1]$:
\begin{equation}
\left(\frac{\lambda(1-\lambda)(1+\gamma\beta_k)}{2\beta_k}-\eta\right)d^2(x_{k+1},x^*) \leq \frac{\lambda}{2\beta_k}d^2(x_k,x^*) - \left(\frac{\lambda}{2\beta_k}-\eta\right)d^2(x_{k+1},x_k).\tag*{$(\circ)_1$}\label{fejer_less_easy_my_way_uni}
\end{equation}
In particular (setting $\lambda = \frac{1}{2}$), we have
\begin{equation}
\left(\frac{1+\gamma\beta_k}{8\beta_k}-\eta\right)d^2(x_{k+1},x^*) \leq \frac{1}{4\beta_k}d^2(x_k,x^*) - \left(\frac{1}{4\beta_k}-\eta\right)d^2(x_{k+1},x_k).\tag*{$(\circ)_2$}\label{fejer_less_easy_uni}
\end{equation}
\end{lemma}
\begin{proof}
With Lemma \ref{fejer}, we obtain that for all $k \in \N$, inequality \ref{fejer_easy} or inequality \ref{fejer_less_easy_my_way} holds. The result then follows from the fact that \ref{fejer_easy} actually implies \ref{fejer_less_easy_my_way} as 
\[
\left(\frac{\lambda(1-\lambda)(1+\gamma\beta_k)}{2\beta_k}-\eta\right) < \left(\frac{\lambda(1-\lambda)(1+\gamma\beta_k)}{2\beta_k}\right)  \text{ and } \left(\frac{\lambda}{2\beta_k}-\eta\right) < \left(\frac{\lambda}{2\beta_k}\right)
\]
for all $\lambda \in [0,1]$ and $k \in \N$.
\end{proof}

From Lemma \ref{fejer_uni}, we now immediately obtain the Fej\'er monotonicity of $\{x_k\}_k$ w.r.t.\ $S(f)$. Over Euclidean spaces, this property is established in \cite[Proposition 3.5]{IusemLara2021} under the condition $\{\beta_k\}_k \subseteq (\lowerbound, \upperbound)$. However it is easy to see that the same holds true already for $\{\beta_k\}_k \subseteq [\lowerbound, \upperbound]$.

\begin{lemma}\label{fejeractually}
Let $f$ be a bifunction satisfying $(Ai)$ with $i=2,4,5,6$. Let $\{\beta_k\}_k \subseteq [\lowerbound,\upperbound]$. Suppose that \eqref{PPA} is well-defined and let $\{x_k\}_k$ be its generated sequence. Let $x^*\in S(f)$. Then, for every $k \in \N$:
\[
d^2(x_{k+1}, x^{*}) \leq d^2(x_k,x^{*}).
\]
\end{lemma}
\begin{proof}
Let $x^{*} \in S(f)$ and fix $k \in \N$. Using Lemma \ref{fejer_uni}, we know that \ref{fejer_less_easy_uni} holds. Using $\beta_k \leq \frac{1}{4 \eta},$ we have
\[
\Big(\frac{1+\gamma\beta_k}{8 \beta_k}-\eta\Big)d^2(x_{k+1}, x^{*}) \leq \frac{1}{4 \beta_k}d^2(x_k,x^{*}) - \Big(\frac{1}{4\beta_k}-\eta\Big)d^2(x_{k+1},x_k) \leq \frac{1}{4 \beta_k}d^2(x_k,x^{*}).
\]
With $\beta_k \geq \frac{1}{\gamma-8\eta}$, it follows that
\[
\Big(\frac{1+\gamma\beta_k}{8 \beta_k} - \eta\Big) d^2(x_{k+1},x^{*}) \geq \Big(\frac{1+\gamma\beta_k}{8 \beta_k} - \frac{\gamma\beta_k-1}{8 \beta_k}\Big) d^2(x_{k+1},x^{*})= \frac{1}{4\beta_k}d^2(x_{k+1},x^{*}).
\]
Multiplying with $4\beta_k$ yields the claim.
\end{proof}

Key for our approach towards a quantitative convergence theorem for \eqref{PPA} is an effective estimate for the asymptotic behavior of $d(x_k, x_{k+1})$. This will here take the form of an explicit bound $\alpha: (0, \infty) \rightarrow \N$ such that for any $\varepsilon>0$, there exists a $k \leq \alpha(\varepsilon)$ with $d(x_k, x_{k+1}) < \varepsilon$. Essential for deriving such an $\alpha$ is \ref{fejer_less_easy_my_way_uni} from Lemma \ref{fejer_uni}, by which we obtain an inequality of the form
\[
e_k d^2(x_{k+1}, x_k) \leq c_k d^2(x_k, x^{*}) - d_k d^2(x_{k+1}, x^{*})
\]
where
\[
e_k := \frac{\lambda_k}{2\beta_k}-\eta, \qquad c_k := \frac{\lambda_k}{2\beta_k}, \qquad d_k := \frac{\lambda_k(1-\lambda_k)(1+\gamma\beta_k)}{2\beta_k}-\eta
\]
for any choice of $\lambda_k \in \big[0,1\big]$. In order to derive an effective asymptotic estimate on $d(x_k, x_{k+1})$ from this we, on the one hand, need to argue that $c_k d^2(x_k, x^{*}) - d_k d^2(x_{k+1}, x^{*})$ gets arbitrarily small. Intuitively, we thus, e.g., want that $c_k \leq d_k$ and if $\{\beta_k\}_k \subseteq (\lowerbound,\upperbound)$, the \textit{only} constant choice of $\lambda_k$ that ensures this is $\lambda_k = \frac{1}{2}$ (as in \cite{IusemLara2021}). However, to derive an estimate on $d^2(x_{k+1}, x_k)$ from this, we further, on the other hand, need that $e_k$ can be bounded away from zero for all $k \in \N$, i.e.\ $(\frac{\lambda}{2\beta_k}-\eta) > a$ for some $a > 0$ and all $k \in \N$.  However, having set $\lambda_k = \frac{1}{2}$, this is impossible as $\big(\frac{1}{4 \beta_k}-\eta\big)$ goes to $0$ if for example $\beta_k \rightarrow \upperbound$, as is permitted in the condition $\{\beta_k\}_k \subseteq (\lowerbound,\upperbound)$.

We will overcome this issue by showing that if $\lambda_k$ is chosen adaptively with $k$ (or rather, the values of $\beta_k$), we can ensure both properties, that is both $c_k - d_k \leq 0$ and the existence of a uniform positive lower bound on $e_k$, at the same time. This not only enables our quantitative results, but also allows us to strengthen the original convergence result of Iusem and Lara (recall Theorem \ref{ILconv} above) to the assumption that $\beta_k \in [\lowerbound, \upperbound ]$ for all $k \in \N$.\footnote{We further refer to the extensive discussions in the master thesis of the first author \cite{Despres2026} for why this extension is also crucial from the logical perspective of the proof mining program, through which the present results have been obtained, to allow for the extraction of rates as in Theorem \ref{THEOREM2} later on.}
	
Before we give these quantitative results in the following part of this section, we need to develop some further machinery. Once this is established, the quantitative results for the asymptotic behavior of $d(x_k, x_{k+1})$ will follow from the following abstract result derived from \cite[Lemma 4.2]{Pischke2025}:

\begin{lemma}\label{easy_folklore} Let $\{a_k\}_k$ be a nonincreasing sequence of nonnegative reals and let $b\geq a_0$. Let further $\{c_k\}_k$ and $\{d_k\}_k$ be  sequences of nonnegative reals with $c_k \leq d_k$ for all $k \in \N$ and such that there exists $u \in \R$ with $c_k \leq u $  for all $k \in \N$.  Then, for any $\varepsilon > 0$:
\[
\exists k \leq \Bigg \lceil \frac{bu}{\varepsilon} \Bigg \rceil \left(c_ka_k - d_ka_{k+1} < \varepsilon\right).
\]
\end{lemma}
\begin{proof}
For any $k\in\mathbb{N}$, we have
\[
c_k a_k - d_k a_{k+1}\leq c_k(a_k - a_{k+1})\leq u(a_k - a_{k+1}).
\]
The result now follows from \cite[Lemma 4.2]{Pischke2025} applied to $\{ua_k\}_k$.
\end{proof}

For the rest of this section, we now suppose that $S(f) \neq \emptyset$ and denote the unique element of $S(f)$ by $x^*$. Similarly, we will presume that, given a suitable sequence $\{\beta_k\}_k$, \eqref{PPA} is well-defined and generates a sequence $\{x_k\}_k$.

By Theorem \ref{nonemptyness_thrm}, the former is guaranteed under the conditions $(A0)$, $(A1)$, $(A2)$ and $(A4)$ and by Proposition \ref{Lemma10}, the latter is guaranteed under the conditions $(A1)$, $(A4)$ and $(A5)$, and whenever $S(f) \neq \emptyset$. However, as we will see, once these two properties are assumed, the remainder of the convergence proof does not rely on $(A1)$ (see in particular Remark \ref{firstAlgStrengthening} later on).

As outlined before, key for establishing this result over the broader set of parameters $[\lowerbound, \upperbound]$ is to choose $\lambda$ in \ref{fejer_less_easy_my_way_uni} of Lemma \ref{fejer_uni} adaptively as $\lambda_{k}$, depending on the values of $\beta_{k}$. The key function that later allows us to compute $\lambda_k$ from $\beta_k$ is now the following: For $x \geq \lowerbound$, set
\begin{equation}
N(x)  := \frac{\gamma x}{2(1+\gamma x)} + \frac{\sqrt{\gamma^2 x^2 - 8 x\eta(1+\gamma x)}}{2(1+\gamma x)},\tag{$\dagger$}\label{N}
\end{equation}
where $\gamma$ and $\eta$ are as in $(A4)$ -- $(A6)$. Indeed, notice that $N$ is well-defined on $[\lowerbound,\infty)$, as for $x \geq \lowerbound$ the value $\gamma^2x^2 - 8x\eta(1+\gamma x)$ is positive. The key properties of $N$ are now as follows:

\begin{lemma}\label{Nporps}
Consider $N$ as defined in \eqref{N}. Then, we have the following:
\begin{enumerate}
\item[(i)] $N(\lowerbound) = \frac{1}{2},$
\item[(ii)] $N(\upperbound) > \frac{4\gamma}{3(8\eta+2\gamma)} > \frac{1}{2}$,
\item[(iii)] $N: [\lowerbound, \infty) \rightarrow [0,1]$ is increasing,
\item[(iv)] $\frac{N(x)}{2x}$ is decreasing on $[\lowerbound, \upperbound]$.
\end{enumerate}
\end{lemma}

\begin{proof}
\begin{enumerate}
\item[(i)] Notice that $\sqrt{\gamma^2 x^2 - 8x\eta(1+\gamma x)} = x \cdot \sqrt{\gamma^2 - 8\eta / x - 8\gamma\eta}$. Hence
\[
N\Big(\lowerbound\Big) = \frac{\gamma}{2(\gamma-8\eta)+2\gamma}+\frac{\sqrt{\gamma^2-8\eta(\gamma-8\eta)-8\gamma\eta}}{2(\gamma-8\eta)+2\gamma} = \frac{\gamma + \sqrt{(\gamma-8\eta)^2}}{4\gamma-16\eta} = \frac{1}{2}.
\]
\item[(ii)] Using $(A6)$, we obtain
\begin{align*}
N\Big(\upperbound\Big)&=\frac{\gamma}{8\eta+2\gamma}+\frac{\sqrt{\gamma^2-32\eta^2-8\eta\gamma}}{8\eta+2\gamma} \overset{12\eta<\gamma}{>}\frac{\gamma}{8\eta+2\gamma}+\frac{\sqrt{\gamma^2-\frac{2}{9}\gamma^2-\frac{6}{9}\gamma^2}}{8\eta+2\gamma} \\[2ex]
&= \frac{\gamma+\frac{1}{3}\gamma}{8\eta+2\gamma}\overset{12\eta<\gamma}{>}\frac{\gamma+4\eta}{8\eta+2\gamma} = \frac{1}{2}.
\end{align*}
\item[(iii)] Notice that
\begin{equation*}
N(x) \leq \frac{\gamma x}{2(1+\gamma x)} + \frac{\sqrt{\gamma^2 x^2 }}{2(1+\gamma x)} = \frac{2\gamma x}{2 + 2\gamma x} < 1
\end{equation*}
for all $x \geq 1/(\gamma-8\eta)$. Together with (i), this implies that $N: [\lowerbound, \infty) \rightarrow [0,1]$. Notice further that
\begin{equation}
N(x) = \frac{\gamma}{\frac{2}{x}+2\gamma}+ \frac{\sqrt{\gamma^2-\frac{8\eta}{x}-8\gamma\eta}}{\frac{2}{x}+2\gamma}.\tag{$-$}\label{Nbeta}
\end{equation}
Form this, it is clear that  $N: [\lowerbound, \infty) \rightarrow [0,1]$ is increasing.
\item[(iv)] We show that the derivative of $\frac{N(x)}{2x}$ is negative for $x \in \Big[\lowerbound, \upperbound\Big].$ With \eqref{Nbeta}, we obtain
\[
\Big(\frac{N(x)}{2x}\Big)'(x) = -\frac{\gamma^2}{4(1+\gamma x)^2} + \frac{4\eta}{\sqrt{\gamma(\gamma-8\eta)-\frac{8\eta}{x}} \cdot 4 (1+\gamma x) x^2}-\frac{\sqrt{\gamma(\gamma-8\eta)-\frac{8\eta}{x}}\cdot \gamma}{4(1+\gamma x)^2}.
\]
As $x \geq \lowerbound$, we know $\sqrt{\gamma (\gamma - 8 \eta) - \frac{8\eta}{x}} \geq (\gamma - 8\eta)$ and so get 
\[
\Big(\frac{N(x)}{2x}\Big)'(x) \leq - \frac{\gamma^2}{4(1+\gamma x)^2}+\frac{4\eta (\gamma - 8\eta)(1+\gamma x)}{ 4(1+\gamma x)^2}-\frac{(\gamma-8\eta)\gamma}{4(1+\gamma x)^2}.
\]
With $x \leq \upperbound$, this yields
\[
\Big(\frac{N(x)}{2x}\Big)'(x) \leq - \frac{\gamma^2}{4(1+\gamma x)^2}+\frac{4\eta (\gamma - 8\eta)(1+\frac{\gamma}{4\eta})}{ 4(1+\gamma x)^2}-\frac{(\gamma-8\eta)\gamma}{4(1+\gamma x)^2}= \frac{-\gamma^2 + 4\eta (\gamma - 8\eta)}{4(1+\gamma x)^2}.
\]
Finally, using $12 \eta < \gamma$, we obtain
\[
\Big(\frac{N(x)}{2x}\Big)'(x) \leq  \frac{-2/3 \gamma ^2 - 32\eta^2}{4(1+\gamma x)^2}
\]
which is (using that $x\geq 0$) negative by $(A6)$.
\end{enumerate}
\end{proof}

Using $N$, we can now give an explicit definition for the adaptive choice of $\lambda_k$. The following lemma provides this definition, together with the essential properties of the corresponding quantities $e_k$, $c_k$ and $d_k$ as derived from Lemma \ref{Nporps}.

\begin{lemma}\label{collection}
Suppose that $(A6)$ holds. For any $k \in \N$, let $\beta_k \in [\lowerbound, \upperbound]$ and consider
\[
e_k := \frac{\lambda_k}{2\beta_k}-\eta, \qquad
	c_k := \frac{\lambda_k}{2\beta_k}, \qquad
	d_k := \frac{\lambda_k(1-\lambda_k)(1+\gamma\beta_k)}{2\beta_k}-\eta
\]
for
\[
\lambda_k:=\mathrm{min}\Big\{N(\beta_k), \frac{1}{2}+\frac{1}{\sqrt{8}}\Big\},
\]
where $N$ is defined as in \eqref{N}. Then, for any $k \in \N$:
\begin{itemize}
\item[(i)] $c_k - d_k \leq 0$,
\item[(ii)] $\lambda_k \in \big[\frac{1}{2},\frac{1}{2}+\frac{1}{\sqrt{8}}\big]$,
\item[(iii)] $e_k \geq (2\gamma\eta-24\eta^2)/(3(8\eta+2\gamma))>0$,
\item[(iv)] $c_k \leq (\gamma-8\eta)/4$.
\end{itemize}
\end{lemma}
\begin{proof}
\begin{itemize}
\item[(i)] We have 
\[
c_k - d_k = \frac{\lambda_k - \lambda_k(1-\lambda_k)(1+\gamma\beta_k)}{2\beta_k}+\eta
\]
and argue that $c_k - d_k = 0$ if $\lambda_k = N(\beta_k)$. To see this, notice that
\[
c_k - d_k = 0 \Leftrightarrow \lambda_k^2 - \lambda_k\frac{\gamma\beta_k}{(1+\gamma\beta_k)}+\frac{2\beta_k\eta}{(1+\gamma\beta_k)} = 0.
\]
By definition of $N$, the right hand side now holds by the quadratic formula. Hence, the claim holds if $\lambda_{k} = N(\beta_k)$. It is left to show the result for cases where $\lambda_{k} = \frac{1}{2}+\frac{1}{\sqrt{8}} < N(\beta_k)$. For this, notice that if we regard $c_k-d_k$ as a function in (the values of) $\lambda_k$, it is increasing on the interval $[\frac{1}{2},1]$. Thus, $c_k-d_k \leq 0$ if $\lambda_k = \frac{1}{2}+\frac{1}{\sqrt{8}} < N(\beta_k)$, as $N(\beta_k) < 1$ by Lemma \ref{Nporps}, (iii).
\item[(ii)] From Lemma \ref{Nporps}, (i), we know that $N(\lowerbound) = \frac{1}{2}$. As by Lemma \ref{Nporps}, (iii), we know further that $N$ is increasing, the result follows from the definition of $\lambda_k$.
\item[(iii)] Suppose $\lambda_k = N(\beta_k)$. As $\frac{N(x)}{2x}$ decreases on the interval $[\lowerbound, \upperbound]$ by Lemma \ref{Nporps}, (iv), together with Lemma \ref{Nporps}, (ii) we obtain
\begin{align*}
e_k  &= \frac{N(\beta_k)}{2\beta_k} - \eta\geq \frac{N\Big(\upperbound\Big)}{2}4\eta-\eta \\
&> \frac{8\gamma\eta}{3(8\eta+2\gamma)}-\eta = \frac{8\gamma\eta-6\gamma\eta-24\eta^2}{3(8\eta+2\gamma)} = \frac{2\gamma\eta-24\eta^2}{3(8\eta+2\gamma)}.
\end{align*}
It follows that $e_k > 0$ as $12\eta < \gamma$ by $(A6)$. Suppose that $\lambda_{k} = \frac{1}{2}+\frac{1}{\sqrt{8}}$. But then, we obtain
\[
e_k = \frac{\lambda_k}{2\beta_k}-\eta \geq \frac{1}{2}\Big(\frac{1}{2}+\frac{1}{\sqrt{8}}\Big)4\eta-\eta = \frac{2}{\sqrt{8}}\eta > \frac{2}{3}\eta >   \frac{2\gamma \eta - 24 \eta^2}{3(8\eta + 2\gamma)}.
\]
As before, we have $e_k > 0$ since $12\eta < \gamma$ by $(A6)$. This implies the result.
\item[(iv)] With Lemma \ref{Nporps}, (iv) and Lemma \ref{Nporps}, (i) we obtain
\[
c_k = \frac{\lambda_k}{2\beta_k} \leq \frac{N(\beta_k)}{2\beta_k} \leq \frac{1}{2} N\Big(\lowerbound\Big)(\gamma-8\eta) = \frac{\gamma-8\eta}{4}.\qedhere
\]
\end{itemize}
\end{proof}

Our main quantitative result for the asymptotic behavior of $d(x_k,x_{k+1})$ is now the following:

\begin{lemma}\label{new and better}
Let $f$ be a bifunction satisfying (A$i$) with $i = 2,4,5,6$. Let $\{\beta_k\}_k \subseteq [\lowerbound,\upperbound]$. Suppose that \eqref{PPA} is well-defined and let $\{x_k\}_k$ be its generated sequence. Let $x^*\in S(f)$. Then, for any $\varepsilon >0$:
\[
\exists k \leq \left \lceil \frac{3b^2(\gamma-8\eta)(2\gamma+8\eta)}{8\eta(\gamma-12\eta) \varepsilon}  \right \rceil (d^2(x_{k}, x_{k+1}) < \varepsilon),
\]
where $b^2 \geq d^2(x_0,x^*)$.
\end{lemma}

\begin{proof}
With Lemma \ref{fejer_uni}, we know that \ref{fejer_less_easy_my_way_uni} holds and in particular obtain 
\[
e_k d^2(x_{k+1},x_k) \leq c_k d^2(x_k,x^*) - d_k d^2(x_{k+1},x^*)
\]
for all $k \in  \N$, where $c_k, e_k$ and $d_k$ are defined as in Lemma \ref{collection}. By Lemma \ref{collection}, (i) and (ii), it holds that $0<c_k \leq d_k$ for all $k \in \N$. Further, Lemma \ref{collection}, (iv) implies  that $c_k \leq \frac{\gamma-8\eta}{4}$. We can thus apply Lemma \ref{easy_folklore} to obtain that for any $\epsilon > 0,$
\[
\exists k \leq \left \lceil \frac{b^2(\gamma-8\eta) }{ 4\varepsilon}  \right \rceil \left(e_k d^2(x_{k+1},x_k) < \varepsilon\right).
\]
With Lemma \ref{collection}, (iii), we now obtain
\[
\exists k \leq \left \lceil \frac{b^2 (\gamma-8\eta)}{ 4\varepsilon} \right \rceil  \left(\frac{2\gamma\eta-24\eta^2}{3(8\eta+2\gamma)} d^2(x_{k+1},x_k) < \varepsilon\right)
\]
for any $\varepsilon>0$. This implies the result.
\end{proof}

With this, we obtain a quantitative convergence theorem that generalizes the assumptions of Theorem \ref{ILconv}. In particular, the quantitative result we provide comes in the form of a sublinear non-asymptotic guarantee which only depends on the parameters $\gamma, \eta$ as in $(A4)$ -- $(A6)$, as well as a bound $b > 0$ on the distance of the starting point of \eqref{PPA} to the equilibrium point $x^{*}$.

\begin{theorem}\label{THEOREM2}
Let $X$ be a Hadamard space and let $f: X \times X \rightarrow \R$ be a bifunction satisfying $(Ai)$ with $i=2,4,5,6$. Let $\{\beta_k\}_k\subseteq [ \frac{1}{\gamma - 8\eta},\frac{1}{4 \eta}]$. Suppose that \eqref{PPA} is well-defined and let $\{x_k\}_k$ be its generated sequence. Suppose that $S(f) \neq \emptyset$. Then $\{x_k\}_k$ converges to the (unique) equilibrium point $x^*$ of $f$. Further, we have the non-asymptotic guarantee 
\begin{equation*}
\forall k > 2 \left(d(x_k, x^{*}) < \frac{C}{\sqrt{k-2}}\right)
\end{equation*}
where
\begin{equation*}
C := \frac{12\sqrt{3}b\gamma^2}{\sqrt{\eta}(\gamma-12\eta)^{3/2}}
\end{equation*}
with $b^2 \geq d^2(x_0,x^{*})$.
\end{theorem}
\begin{proof}
It is enough to show the quantitative result. To that end, we first show
\begin{equation*}
\forall \varepsilon > 0 \ \forall k \geq \left( \left \lceil \frac{C^2}{\varepsilon^2} \right\rceil + 1\right) \left(d(x_k, x^{*}) < \varepsilon\right).\tag{Q}\label{firstQuant}
\end{equation*}
Given $\varepsilon >0$, using Lemma \ref{new and better}, take
\begin{equation*}
k \leq \left\lceil  \frac{3b^2(\gamma-8\eta)(2\gamma+8\eta)}{8\eta(\gamma-12\eta) } \cdot \frac{64 (3\gamma-16\eta)^2}{(3\gamma-28\eta)^2 \varepsilon^2}  \right \rceil\leq \left\lceil \frac{C^2}{\varepsilon^2}\right\rceil
\end{equation*}
such that 
\begin{equation*}
d(x_{k+1},x_k) <\frac{3\gamma-28\eta}{8(3\gamma-16\eta)} \varepsilon.
\end{equation*}
Notice that as $12 \gamma - 112 \eta \leq 24\gamma - 128 \eta$ by $(A6)$, we in particular obtain that $d(x_{k+1} , x_k ) < \frac{1}{4} \varepsilon$. Using that $x_{k+1} \in \mathrm{Prox}_{\beta_k f(x_k,\cdot)}(x_k)$, applying Lemma \ref{Lemma11P} to the function $f(x_k,\cdot)$ with $y = x^{*}$ yields
\begin{align*}
&f(x_k, x_{k+1}) - \mathrm{max}\{f(x_k, x_{k+1}), f(x_k,x^*)\}\\
&\qquad\leq   \frac{\lambda}{2}\Big(\frac{\lambda}{\beta_k} - \gamma + \lambda\gamma\Big)d^2(x_{k+1},x^*) + \frac{\lambda}{\beta_k} \Big\langle \vv{x_kx_{k+1}},\vv{x_{k+1}x^*}\Big\rangle
\end{align*}
for all $\lambda \in [0,1]$. We consider two cases:

\textbf{Case 1:} $f(x_k,x_{k+1}) \geq f(x_k, x^*)$. We get
\begin{equation*}
0 \leq \frac{\lambda}{2}\left(\frac{\lambda}{\beta_k} - \gamma + \lambda\gamma\right)d^2(x_{k+1},x^*) + \frac{\lambda}{\beta_k}\svv{x_k}{x_{k+1}}{x_{k+1}}{x^*}.
\end{equation*}
for all $\lambda \in [0,1]$. Multiplying by $\frac{\beta_k}{\lambda}$, using \eqref{cauchy} and the fact that $\beta_k \geq \lowerbound$, we obtain
\begin{align*}
&\frac{1}{2}\left((1-\lambda)\frac{\gamma}{(\gamma-8\eta)}-\lambda\right)d^2(x_{k+1},x^*) \\
&\qquad\leq \svv{x_k}{x_{k+1}}{x_{k+1}}{x^*} \leq d(x_{k+1},x_k)d(x_{k+1},x^*)
\end{align*}
for all $\lambda \in (0,1]$. Now either $d(x_{k+1},x^{*}) = 0$, and we are done by the Fej\'er monotonicity of $\{x_k\}_k$ w.r.t.\ $S(f)=\{x^{*}\}$ (recall Lemma \ref{fejeractually}), or $d(x_{k+1},x^{*}) > 0$, and we obtain
\[
\frac{1}{2}\left((1-\lambda)\frac{\gamma}{\gamma-8\eta}-\lambda\right)d(x_{k+1},x^*) \leq d(x_{k+1},x_k)
\]
for all $\lambda \in (0,1]$. This in particular holds for $\lambda = \frac{1}{8}$. By the above, we now have that
\[
d(x_{k+1},x_k)< \frac{1}{4}\varepsilon  = \frac{3\gamma + 4\eta}{12\gamma + 16\eta} \varepsilon<\frac{3\gamma + 4 \eta}{8(\gamma - 8 \eta)} \varepsilon.
\]
As $(3\gamma + 4 \eta)/(8(\gamma - 8 \eta)) > 0$ by $(A6)$, we get $d(x_{k+1},x^{*})< \varepsilon$ and thus, with Lemma \ref{fejeractually}, we obtain $d(x_j,x^{*}) < \varepsilon$ for all $j\geq k+1$.

\textbf{Case 2:} $f(x_k,x_{k+1}) < f(x_k, x^*)$. Then, just as in the proof of Lemma \ref{fejer} in that case, we obtain
\begin{align*}
0 &\leq  \frac{\lambda}{2}\left(\frac{\lambda}{\beta_k} - \gamma + \lambda\gamma\right)d^2(x_{k+1},x^*) + \frac{\lambda}{\beta_k}\svv{x_k}{x_{k+1}}{x_{k+1}}{x^*} \\
&\hphantom{\leq\mbox{}}+ \eta\Big(d^2(x_{k+1},x_k) + d^2(x_{k+1},x^*)\Big)
\end{align*}
for all $\lambda \in [0,1].$ Similar to Case 1, multiplying by $\frac{\beta_k}{\lambda}$ and using \eqref{cauchy}, we obtain
\begin{align*}
&\frac{1}{2}\left((1-\lambda)\gamma \beta_k-\lambda - \frac{2 \eta \beta_k}{\lambda}\right)d^2(x_{k+1},x^*) \\
&\qquad\leq d(x_{k+1},x_k)d(x_{k+1},x^*) + \frac{\eta \beta_k}{\lambda}d^2(x_{k+1},x_k)
\end{align*}
for all $\lambda \in (0,1]$. We make two further case distinctions:

\textbf{Case 2a:} $d(x_{k+1},x^{*}) \leq d(x_{k+1},x_k)$. As with $d(x_{k+1} , x_k ) < \frac{1}{4} \varepsilon$, we in particular have that $d(x_{k+1},x_k) < \varepsilon$, the result follows with Lemma \ref{fejeractually}.

\textbf{Case 2b:} $d(x_{k+1},x^{*}) > d(x_{k+1},x_k)$. This yields
\begin{equation*}
\frac{1}{2}\left((1-\lambda)\gamma \beta_k-\lambda - \frac{2\eta \beta_k}{\lambda}\right)d(x_{k+1},x^*)< d(x_{k+1},x_k) + \frac{\eta \beta_k}{\lambda} d(x_{k+1},x_k).
\end{equation*}
for all $\lambda \in (0,1]$. Multiplying both sides by $\frac{\lambda}{\lambda+\beta_k \eta}$ and rewriting the expression slightly, we thus get
\begin{equation*}\label{sigh}
\frac{1}{2(\frac{\lambda}{\beta_k} +  \eta)}\left((\lambda(1-\lambda)\gamma-2\eta)-\frac{\lambda^2}{\beta_k}\right) d(x_{k+1},x^{*})< d(x_{k+1},x_k)
\end{equation*}
for all $\lambda \in (0,1]$. We now want to bound the factor in front of $d(x_{k+1}, x^{*})$ from below. To this end, we set $\lambda = \frac{3}{8}$ and, using $\beta_k \geq \lowerbound$, obtain
\begin{equation*}
(\lambda(1-\lambda)\gamma-2\eta)-\frac{\lambda^2}{\beta_k} \geq (\lambda (1-\lambda) \gamma - 2 \eta) - \lambda^2(\gamma-8\eta)  = \frac{6\gamma -56 \eta}{64} >0
\end{equation*}
with $(A6)$. Then further, again as $\beta_k \geq \lowerbound$, we get
\begin{equation*}
\frac{1}{2(\frac{\lambda}{\beta_k} +  \eta)}\Bigg((\lambda(1-\lambda)\gamma-2\eta)-\frac{\lambda^2}{\beta_k}\Bigg) \geq \frac{1}{2\bigg(\frac{3(\gamma-8\eta)}{8}+\eta\bigg)}\bigg(\frac{6\gamma-56 \eta}{64}\bigg)= \frac{3\gamma - 28 \eta}{8(3\gamma-16\eta)}.
\end{equation*}
Together, the above yields
\[
\frac{3\gamma-28\eta}{8(3\gamma-16\eta)} d(x_{k+1}, x^{*})< d(x_{k+1},x_k) < \frac{3\gamma-28\eta}{8(3\gamma-16\eta)}\; \varepsilon
\]
and since $(3\gamma-28\eta)/8(3\gamma-16\eta)> 0$, we get $d(x_{k+1},x^{*}) < \varepsilon$. Using Lemma \ref{fejeractually}, we in particular have $d(x_j,x^{*}) < \varepsilon$ for any $j \geq k+1$.

This completes the proof of \eqref{firstQuant}. To obtain the non-asymptotic variant of that effective convergence result, set $\varepsilon := C/\sqrt{k-2}$ given a $k>2$, and rearrange.
\end{proof}

\begin{remark}\label{firstAlgStrengthening}
Already over $\mathbb{R}^d$, the above result strengthens the convergence result of Iusem and Lara (recall Theorem \ref{ILconv}) in three ways: First, the above result provides quantitative estimates that (to our knowledge) are already novel in the Euclidean case. Furthermore, the assumption that the parameter sequence $\{\beta_k\}_k$ is contained in the open interval $(\lowerbound, \upperbound)$ was weakened to the closed interval $[\lowerbound, \upperbound]$. Lastly, the above result allows for a careful examination of the main assumptions on the bifunction. In that context, we find that $(A3)$ can be dropped completely and, if the existence of an equilibrium point and the well-definedness of the algorithm are assumed beforehand, also $(A1)$ can be dropped. If, on the other hand, $(A1)$ is added as an assumption to Theorem \ref{THEOREM2} above, then the assumptions that $S(f) \neq \emptyset$ and that the algorithm is well-defined can both be dropped, using Theorem \ref{nonemptyness_thrm} and Proposition \ref{Lemma10} respectively.
\end{remark}

\begin{remark}\label{rem:Fejer}
The quantitative estimate contained in Theorem \ref{THEOREM2} above could also have been obtained by an application of the general results on rates of convergence for Fej\'er monotone sequences under general metric regularity assumptions as developed in the work of Kohlenbach, L\'opez-Acedo and Nicolae \cite{KohlenbachLN2019}. We refer to the master thesis of the first author \cite{Despres2026} for a detailed discussion in that vein.
\end{remark}

\begin{remark}[For logicians]\label{rem:Logic}
As the assumption $(A2)$, that is pseudomonotonicity, is a (generalized) $\Pi_2$-statement, the logical methodology of proof mining underlying the quantitative result contained in Theorem \ref{THEOREM2} would a priori suggest a dependence of the respective rate on a (uniform) quantitative rendering of this pseudomonotonicity of $f$, say in the form of a \emph{modulus of uniform pseudomonotonicity} for $f$, that is for a fixed $ o \in X$ a function $p: \N \times \N \rightarrow \N$ such that
\begin{equation*}
f(x,y) \geq -\frac{1}{p(b,n) +1} \rightarrow f(y,x) \leq \frac{1}{n+1}
\end{equation*}
for all $b, n \in \N$ and all $x,y \in X$ with $d(o,x),d(o,y) \leq b$. By the general logical metatheorems of proof mining for metric and $\CAT$-spaces \cite{KohlenbachGerhardy07}, an effective modulus of uniform pseudomonotonicity can be extracted from a corresponding (suitable) proof of pseudomonotonicity. Indeed, as an example for this, consider $f_h(x,y):= h(y) - h(x)$ as defined in Example \ref{ex_min_is_equi} for some lsc and strongly quasiconvex $h: X \rightarrow \R$, where it can be rather immediately seen that a modulus of uniform pseudomonotonicity is given by $p(b,n) := n$.

However, the extracted rate of convergence from Theorem \ref{THEOREM2} \emph{does not} depend on such a modulus and in fact this independence can be logically explained and guaranteed a priori. Indeed, in the context of Theorem \ref{THEOREM2}, the assumption $(A2)$ is only ever used to deduce $\forall x\in X\left(f(x, x^{*}) \leq 0\right)$ from the characterizing property $\forall x\in X \left(f(x^{*},x ) \geq 0\right)$ of a given (but fixed) solution $x^*$. Instead of $(A2)$, it hence suffices to work only with the two properties above, formulated relative to such a fixed solution $x^*$. As these are universal properties, they do not contribute to the extracted bounds, which guarantees the independence of the resulting rate from any such modulus of uniform pseudomonotonicity a priori. We refer to the master thesis of the first author \cite{Despres2026} for a more in-depth discussion of the surrounding logical aspects.
\end{remark}

We end this section by mentioning two further quantitative results that allow for a simplification of the constant $C$ featuring in Theorem \ref{THEOREM2} under additional assumptions on the parameters. Concretely, as $\{\beta_k\}_k$ can be freely chosen, and since $(\lowerbound, \upperbound)$ is nonempty, a permissible restriction for applying \eqref{PPA} would be to choose $\{\beta_k\}_k$ from a compact subinterval of $(\lowerbound, \upperbound)$. If one follows this line, both the constant $C$ and the resulting arguments for the convergence of the method simplify drastically, to the point where the adaptive choice of $\lambda_k$ depending on $\beta_k$, as facilitated above by the function $N$ from \eqref{N}, is no longer necessary. We collect the resulting quantitative estimates in the following theorem, but omit the (simplified) proofs for them, which can be found in the master thesis of the first author \cite{Despres2026}.

\begin{theorem}\label{THEOREMa}
Let $X$ be a Hadamard space and let $f: X \times X \rightarrow \R$ be a bifunction satisfying $(Ai)$ with $i=2,4,5,6$. Let $\Theta >0$ with $\lowerbound < \Theta < \upperbound$ and let $\{\beta_k\}_k\subseteq [\lowerbound,\Theta]$. Suppose that \eqref{PPA} is well-defined and let $\{x_k\}_k$ be its generated sequence. Suppose that $S(f) \neq \emptyset$. Then $\{x_k\}_k$ converges to the (unique) equilibrium point $x^*$ of $f$. Further, we have the non-asymptotic guarantee
\begin{equation*}
\forall k > 2  \left(d(x_k, x^{*}) < \frac{C'}{\sqrt{k-2}}\right)
\end{equation*}
where
\begin{equation*}
C':= \frac{24b \gamma}{\sqrt{\Theta^*}(3\gamma-8\eta)}
\end{equation*}
with $b^2 \geq d^2(x_0,x^{*})$ and $\Theta^* := 1-4\eta \Theta$. If we further have a $\theta>0$ such that $\{\beta_k\}_k \subseteq [\theta,\Theta] \subseteq ( \frac{1}{\gamma - 8\eta},\frac{1}{4 \eta})$, then we can set
\[
C' := \mathrm{max}\left\{1, \frac{3}{2 \theta^*}\right\} \frac{b }{\sqrt{\Theta^*}}
\]
with $b$ and $\Theta^*$ as before as well as $\theta^* :=  (\gamma-8\eta)\theta-1$.
\end{theorem}
	
\section{Regularized selection methods}\label{subsec_alog2}

\subsection{The method and preliminaries}

We now move to the second method for solving the equilibrium problem \eqref{equil} over a Hadamard space $(X,d)$ for a given bifunction $f: X\times X \rightarrow \R$, now following the first type of method discussed in the introduction which regularizes the bifunction as a whole (yielding the notion of a resolvent of the bifunction). We recall it here first over a Hilbert space $(X,\langle\cdot,\cdot\rangle)$ and a given closed and convex set $K\subseteq X$. Given a starting point $x_0\in K$ and a sequence $\{\beta_k\}_{k} \subseteq (0,\infty)$, recursively define $x_{k+1}$ by choosing
\[
x_{k+1} \in S(K,f_k)\text{ for }f_k(x,y) := f(x,y) + \frac{1}{\beta_k} \langle x-x_k,y-x\rangle,\tag{\%}\label{preRS}
\]
where $S(K,f_k)$ refers to the solution set of the equilibrium problem defined by $f_k$ over $K$, similar to before. Over Euclidean spaces, Iusem and Lara show that the sequence $\{x_k\}_k$ generated by this algorithm converges to an equilibrium point in the context of a suitable strongly quasiconvex and pseudomonotone bifunction:

\begin{theorem}[{cf.\ \cite[Corollary 3.2]{IusemLara2021}}]\label{ILconv2}
Let $X = \R^d$ and $K \subseteq X$ be closed and convex. Let $f:K\times K\to\mathbb{R}$ be a bifunction satisfying $(A1)$ -- $(A5)$. Take $\beta >0$ such that $\beta_k \geq \beta$ for all $k \in \N$ and let $\{x_k\}_k$ be the sequence generated by \eqref{preRS} over $K$. Then $\{x_k\}_k$ converges to a point $x^{*} \in S(K,f)$. 
\end{theorem}

In fact, in \cite[Proposition 3.6]{IusemLara2021}, Iusem and Lara further ``show'' that \eqref{preRS} is well-defined, i.e.\ that $S(K,f_k)\neq\emptyset$ for all $k\in\mathbb{N}$. However, we cannot confirm their result. The following remark provides a detailed discussion in that vein, and in particular provides a counterexample to their main line of reasoning.

\begin{remark}\label{counter}
In the proof of the well-definedness of the method generated by \eqref{preRS} under the assumptions of Theorem \ref{ILconv2} as given in \cite[Proposition 3.6]{IusemLara2021}, it is argued that the nonemptyness of $S(K,f_{z, \beta})$ for the function $f_{z, \beta}(x,\cdot) := f(x, \cdot) + \frac{1}{\beta}\langle x-z, \cdot -x \rangle$ (and hence the nonemptyness of $S(K,f_k)$) follows from the fact that $f_{z, \beta}(x,\cdot)$ attains its minimum on $K$ for any $z,x\in K$ and $\beta >0$. However it seems like this implication should not hold. Towards a counterexample, we utilize \cite[Example 4.2]{IusemLara2021}. Take $ \alpha > 0$ and $K = [0, 2]$. Define the bifunction $f: K \times K \rightarrow \R$ by 
\[
f(x,y) := \sqrt{y} - \sqrt{x} + \langle \alpha x , y - x\rangle.
\]
Note that $\langle a,b \rangle := ab$ for all $a,b \in \R$. If follows by \cite[Example 4.2]{IusemLara2021} that the assumptions $(A1)$ -- $(A6)$ hold for this function. Also, the function
\[
f_{z,\beta}(x,\cdot):= \sqrt{\cdot} - \sqrt{x} + \langle (\alpha +1/\beta)x -z/\beta , \cdot - x\rangle
\]
clearly attains a minimum on $K = [0, 2]$ for any $z,\beta$ and $x$. Now, consider $\beta = 1$ and $z = 1$. We will show that $S(K,f_{1,1})=\emptyset$. For a contradiction, assume that $x^{*} \in S(K,f_{1,1})$. Then by assumption
\[
f_{1,1}(x^{*},y) = \sqrt{y}-\sqrt{x^{*}} + \langle (1+\alpha) x^{*} -1, y - x^{*} \rangle  \geq 0
\]
for all $y \in K$. Suppose that $ x^{*} = 0$. Then, for $y = 2$, we obtain
\[
f_{1,1}(0,2) = \sqrt{2} + \langle-1, 2 \rangle = \sqrt{2} - 2 < 0,
\]
contradicting that $x^{*} \in S(K,f_{1,1})$. However if $x^{*} > 0$, we obtain for $y = 0$ that
\[
f_{1,1} (x^{*},0) = -\sqrt{x^{*}} + \left((1+\alpha)x^{*} - 1\right) (-x^{*}) = - \sqrt{x^{*}}-(1+\alpha)(x^{*})^2+ x^{*}.
\]
But now, if $x^{*} \leq 1$ we have $\sqrt{x^{*}} \geq x^{*}$ and if $x^{*} > 1$, we have $(x^{*})^2 > x^{*}$. Thus, the above expression is strictly negative for all choices of $x^{*} > 0$. We conclude that $S(K,f_{1,1})$ must be empty.
\end{remark}

Nevertheless, we now move on towards generalizing the method \eqref{preRS} and its convergence result contained in Theorem \ref{ILconv2} to the nonlinear setting of Hadamard spaces, and outfitting it with quantitative information. So, fix a Hadamard space $(X,d)$ and a bifunction $f:X\times X\to\mathbb{R}$. Similar to before, as closed and convex subsets of Hadamard spaces are again Hadamard spaces, we will (without loss of generality) not consider a relativization onto such sets.

There are now multiple ways to generalize the method \eqref{preRS}, based on different ways of regularizing $f$ in a Hadamard space: Concretely, two immediate regularizations corresponding to the above definition in Hilbert spaces are
\[
f_k(x,y) := f(x,y) + \frac{1}{\beta_k} \svv{x_k}{x}{x}{y}
\]
on the one hand, using the quasi-inner product on $X$, as well as
\[
f_k(x,y) := f(x,y) - \frac{1}{\beta_k} g_x(\log_xx_k,\log_x y)
\]
on the other, using the pseudo-inner product on $T_xX$. Yet another possible regularization utilizes the Busemann function of an associated geodesic ray: Supposing that $X$ has the geodesic extension property, one can consider
\[
f_k(x,y) := f(x,y) - \frac{1}{\beta_k} d(x_k,x)(b_{r_{x_k,x}}(y)-b_{r_{x_k,x}}(x)),
\]
where $b_{r_{x_k,x}}$ is the Busemann function corresponding to the (unique) geodesic ray $r_{x_k,x}$ such that $r_{x_k,x}(0)=x_k$ and $r_{x_k,x}(d(x_k,x))=x$ (which exists since $X$ has the geodesic extension property).\footnote{Since we rely on $X$ having the geodesic extension property for the last type of regularization based on Busemann functions, there would be a gain in generality if we would formulate our equilibrium problem and resulting method over a fixed closed and convex set $K\subseteq X$. However, since the approach we outline here can be trivially extended to such situations, and to not create notational mismatch between the methods, we refrain from introducing such a relativization in the following.}

Indeed, all regularizations coincide with \eqref{preRS} in Hilbert spaces. The first is employed in \cite{KHATIBZADEHMOHEBBI2021} in the context of Hadamard spaces. The second was previously taken as a point of departure in work of Colao, L\'opez, Marino and Mart{\'i}n-M{\'a}rquez \cite{COLAO201261} for equilibrium problems over Hadamard manifolds and is considered here (to our knowledge) for the first time in general Hadamard spaces. The last represents a regularization of a bifunction based on Busemann functions motivated by the recent work of Bento, Cruz Neto, Melo, et al.\ \cite{BentoCruzNetoLopesMeloFilho2024,BentoCruzNetoMelo2022}, also set over Hadamard manifolds, modified here to both adequately extend the resulting object to Hadamard spaces and to allow applications in the context of strongly quasiconvex functions. All have previously however only been considered for convex bifunctions $f$.

Based on this plurality, and to provide a uniform study of (at least) these three methods above, we in the following consider a more abstract regularized selection method via
\[
x_{k+1} \in S(f_k)\text{ for }f_k(x,y) := f(x,y) + \frac{1}{\beta_k} r(x_k,x,y),\tag{RS}\label{RS}
\]
where $r:X^3\to\mathbb{R}$ is an abstract regularization function for which we will assume two key properties for any $u,v,w\in X$:
\begin{enumerate}
\item[$(R1)$] $r(u,v,w)\leq\frac{1}{2}(d^2(w,u)-d^2(w,v)-d^2(u,v))$,
\item[$(R2)$] $r(u,v,w)\leq d(u,v)d(v,w)$.
\end{enumerate}

Before moving on, the next two lemmas quickly note that $(R1)$ and $(R2)$ hold for the two regularization functions used above.

\begin{lemma}
Define $r(u,v,w):=\svv{u}{v}{v}{w}$. Then $r$ satisfies $(R1)$ and $(R2)$.
\end{lemma}
\begin{proof}
By definition, we have 
\[
\svv{u}{v}{v}{w}= \tfrac{1}{2}\left(d^2(w,u)-d^2(w,v)-d^2(u,v)\right)
\]
which yields $(R1)$. Further, by \eqref{cauchy}, we get $\svv{u}{v}{v}{w}\leq d(u,v)d(v,w)$ which is $(R2)$.
\end{proof}

\begin{lemma}
Define $r(u,v,w):=-g_v(\log_vu,\log_v w)$. Then $r$ satisfies $(R1)$ and $(R2)$.
\end{lemma}
\begin{proof}
Lemma \ref{tangentCat0} yields
\[
-g_v(\log_vu,\log_v w)\leq \tfrac{1}{2}\left(d^2(w,u)-d^2(w,v)-d^2(u,v)\right)
\]
which is $(R1)$. Further, recall from Section \ref{tangent} that
\[
-g_v(\log_vu,\log_v w)\leq \norm{\log_vu}_v\norm{\log_v w}_v=d(u,v)d(v,w)
\]
by \eqref{cauchy} on $T_xX$, which is $(R2)$.
\end{proof}

To deal with the last regularization, we first require the following result about Busemann functions.

\begin{lemma}\label{lem:BusemannInequality}
Let $X$ have the geodesic extension property. Then, for any $u,v,w\in X$:
\[
d(u,v)b_{r_{u,v}}(w)\geq\frac{1}{2}\left(d^2(w,v)-d^2(w,u)-d^2(u,v)\right).
\]
\end{lemma}
\begin{proof}
To simplify notation, write $r$ for $r_{u,v}$ and $a$ for $d(u,v)$. Given $t>a$, note that $v=(1-\frac{a}{t})u\oplus \frac{a}{t}r(t)$. Hence, using \eqref{CN}, we get
\[
d^2(w,v)\leq \left(1-\frac{a}{t}\right)d^2(w,u)+\frac{a}{t}d^2(w,r(t))-\frac{a}{t}\left(1-\frac{a}{t}\right)d^2(u,r(t)).
\]
As $r$ is a geodesic ray issuing at $u$, we have $d^2(u,r(t))=t^2$ and so $\frac{a}{t}(1-\frac{a}{t})t^2=a(t-a)$. Rearranging the above and multiplying by $\frac{t}{a}$, we get
\begin{align*}
d^2(w,r(t))&\geq \frac{t}{a}\left(d^2(w,v)-d^2(w,u)\right)+d^2(w,u)+t(t-a)\\
&\geq t^2+ tc+d^2(w,u)
\end{align*}
where $c=\frac{1}{a}\left(d^2(w,v)-d^2(w,u)-a^2\right)$. Hence, we get
\begin{align*}
d(w,r(t))-t&\geq \sqrt{t^2+ tc+d^2(w,u)}-t\\
&=\frac{tc+d^2(w,u)}{\sqrt{t^2+ tc+d^2(w,u)}+t}\geq \frac{tc+d^2(w,u)}{2t+\sqrt{tc+d^2(w,u)}}
\end{align*}
so that we get $b_{r_{u,v}}(w)=\lim_{t\to\infty}(d(w,r(t))-t)\geq\frac{c}{2}$. Multiplying by $a$ gives the result.
\end{proof}

Now, we can recognize the last type of regularization as an instance of our abstract schema.

\begin{lemma}
Let $X$ have the geodesic extension property. Define $r(u,v,w):=-d(u,v)(b_{r_{u,v}}(w)-b_{r_{u,v}}(v))$. Then $r$ satisfies $(R1)$ and $(R2)$.
\end{lemma}
\begin{proof}
It is easy to see that $b_{r_{u,v}}(v)=-d(u,v)$. Hence, using Lemma \ref{lem:BusemannInequality}, we get 
\begin{align*}
d(u,v)(b_{r_{u,v}}(w)-b_{r_{u,v}}(v))&=d(u,v)b_{r_{u,v}}(w)+d^2(u,v)\\
&\geq\frac{1}{2}\left(d^2(w,v)-d^2(w,u)-d^2(u,v)\right)+d^2(u,v)\\
&\geq\frac{1}{2}\left(d^2(w,v)-d^2(w,u)+d^2(u,v)\right)
\end{align*}
so that 
\[
r(u,v,w)\leq \frac{1}{2}\left(d^2(w,u)-d^2(w,v)-d^2(u,v)\right)
\]
which is $(R1)$. For $(R2)$, note that we have 
\[
-d(u,v)(b_{r_{u,v}}(w)-b_{r_{u,v}}(v))=d(u,v)(b_{r_{u,v}}(v)-b_{r_{u,v}}(w))\leq d(u,v)d(v,w)
\]
since $b_{r_{u,v}}$ is nonexpansive (recall Section \ref{subsec_Hadmard_spaces}).
\end{proof}

In comparison to \eqref{preRS}, where the well-definedness of the algorithm was already unclear even in the linear case for strongly quasiconvex functions, this abstract phrasing via \eqref{RS} makes the well-definedness of the method even more precarious, especially in Hadamard spaces and already for convex bifunctions. We will hence in the following generally only assume (and not show) that $S(f_k)\neq\emptyset$ for all $k \in \N$ in the Hadamard space setting, so that the method \eqref{RS} is well-defined. However, we briefly comment on some partial results in that vein in Remark \ref{rem:wellDef} below.

In that way, although the proof of the main convergence theorem and its corresponding quantitative estimates given below (see Theorem \ref{secondQuantThm}) are simpler than their counterparts for \eqref{PPA} from the previous section (recall Theorems \ref{THEOREM2} and \ref{THEOREMa}), the great disadvantage of \eqref{RS} is that it is less clear whether the algorithm is well-defined, and it further does not seem immediate in what respect the corresponding subproblems are easy to solve, as they lack important structure (e.g.\ w.r.t.\ quasiconvexity).

\begin{remark}\label{rem:wellDef}
In \cite{KHATIBZADEHMOHEBBI2021}, Khatibzadeh and Mohebbi consider well-definedness of \eqref{RS} in Hadamard spaces for $r(u,v,w):=\svv{u}{v}{v}{w}$, but where $f$ is convex in its second argument. They encounter a different difficulty compared to the above Remark \ref{counter} in this nonlinear context: They can show that $S(f)$ is nonempty if $f$ is convex in its right argument, is $\theta$-undermonotone (see \cite[p.\ 233]{KHATIBZADEHMOHEBBI2021}) and satisfies $(A0)$ and $(A1)$. However, as the function $y \mapsto \svv{z}{x}{x}{y}$ is not necessarily convex in general Hadamard spaces, neither is $f_{z,\beta}(x,\cdot):=f(x, \cdot) + \frac{1}{\beta}r(z,x,\cdot)$ for the above $r$. Thus, the previous arguments, reliant on the convexity property of $f$, cannot be applied to show nonemptyness of $S(f_{z,\beta})$.  Khatibzadeh and Mohebbi leave it as an open question whether the algorithm is well-defined only assuming $(A0)$, $(A1)$, convexity and $\theta$-undermonotonicity of $f$ and only show nonemptyness of $S(f_{z,\beta})$ under additional assumptions on $f$ (e.g.\ if $f$ is cyclically monotone). It was claimed in \cite[Proposition 2.9]{COLAO201261} (as well as \cite[Proposition 3.4, (ii)]{PapaQuirozOliveira2009}) that the issue of the convexity of the regularization function $r$ is mitigated for $r(u,v,w):=-g_v(\log_vu,\log_v w)$ over Hadamard manifolds, which was used in \cite[Theorem 4.9]{COLAO201261} to establish that $S(f_{z,\beta})$ is nonempty whenever $f$ convex in its right argument, monotone and satisfies some further auxiliary conditions (including a compactness condition). However, as shown in \cite{KristalyLiLopezAcedoNicolae2016}, these results from \cite{COLAO201261,PapaQuirozOliveira2009} are actually false in general and the function $r$ is only convex over linear spaces. The last type of regularization $r(u,v,w):=-d(u,v)(b_{r_{u,v}}(w)-b_{r_{u,v}}(v))$ is actually concave in $w$ since the Busemann function is convex (recall Section \ref{subsec_Hadmard_spaces}). This is contrary to the regularization $d(u,v)b_{r_{v,u}}(w)$ introduced in \cite{BentoCruzNetoLopesMeloFilho2024}, which served as a motivation for the above regularization introduced in this paper. Indeed, the convexity of this regularization is crucially used therein to establish the well-definedness of the method (cf.\ \cite[Theorem 1.1]{BentoCruzNetoLopesMeloFilho2024}). In any case, such convexity results for $r$ would not immediately solve the associated problem in the strongly quasiconvex case as the addition of a strongly quasiconvex and a convex function need not be strongly quasiconvex, so that not even Theorem \ref{nonemptyness_thrm} can be applied without additional considerations. Naturally, if $f_{z,\beta}$ happens to be strongly quasiconvex for all $\beta > 0$ and $z \in X$ in its right argument and further satisfies $(A0)$,$(A1)$ and $(A2)$, then Theorem \ref{nonemptyness_thrm} of course guarantees $S(f_{z,\beta}) \neq \emptyset$.
\end{remark} 

\subsection{Quantitative convergence results}

We now turn to the main focus of this section, the generalized and effective variant of Theorem \ref{ILconv2}. In \cite{IusemLara2021}, Iusem and Lara refer\footnote{They actually reference \cite{IusemSosa03}, but we believe this might be a typo.} to \cite[Theorem 1]{IusemSosa10} where, in a Hilbert space setting, weak convergence of the sequence generated by \eqref{preRS} is shown under the assumptions that $S(f) \neq \emptyset$ and that $f$ is suitably semicontinuous, monotone and convex. They then argue that this can easily be transferred to the strongly quasiconvex setting of Theorem \ref{ILconv2}. Inspired by \cite{Pischke2025}, we follow a slightly different argument which in particular will allow us to show a strong convergence result together with a rate of convergence in the general setting of strongly quasiconvex and pseudomonotone bifunctions on Hadamard spaces.

We begin by adapting \cite[Proposition 4]{IusemSosa10} to our setting, which will allow us to show Fej\'er monotonicity of $\{x_k\}_k$ and derive a quantitative result on the asymptotic behavior of $d(x_{k+1}, x_k)$. The proof is an easy adaptation of the proof of \cite[Proposition 4]{IusemSosa10}.

\begin{lemma}\label{algo2_fejer}
Let $f$ be a bifunction satisfying $(A2)$ and let $r$ satisfy $(R1)$. Let $\{\beta_k\}_k\subseteq (0,\infty)$. Suppose that \eqref{RS} is well-defined and let $\{x_k\}_k$ be its generated sequence. Let $x^{*} \in S(f)$. Then, for every $k \in \N$:
\[
d^2(x_{k+1}, x^{*}) \leq d^2(x_k,x^{*}) - d^2(x_{k+1}, x_k).
\]
\end{lemma}
\begin{proof}
Fix $k \in \N$. As $ x_{k+1} \in S(f_k)$, we obtain
\[
0 \leq f_k(x_{k+1},x^{*}) = f(x_{k+1},x^{*}) + \frac{1}{\beta_k} r(x_k,x_{k+1},x^*)
\]
and thus $- f(x_{k+1},x^{*}) \leq \frac{1}{\beta_k} r(x_k,x_{k+1},x^*)$. As $x^{*}$ is an equilibrium point for $f$, we have $f(x^{*}, y) \geq 0$ for all $y \in X$, and by pseudomonotonicity of $f$ it follows that $f(y, x^{*}) \leq 0$ for all $y \in X$. In particular, we have $f(x_{k+1}, x^{*}) \leq 0$, and hence
\[
0 \leq \frac{1}{\beta_k} r(x_k,x_{k+1},x^*) \leq \frac{1}{2\beta_k}\big(d^2(x^{*}, x_k) - d^2(x^{*}, x_{k+1}) - d^2(x_{k+1}, x_k)\big),
\]
using $(R1)$. As $\beta_k > 0$ for all $k \in \N$, this implies the result.
\end{proof}

The approach towards our quantitative convergence theorem for \eqref{RS} is now similar as with \eqref{PPA}. First, we derive an effective estimate for the asymptotic behavior of $d(x_k, x_{k+1})$ from the above.

\begin{lemma}\label{toolate}
Let $f$ be a bifunction satisfying $(A2)$ and let $r$ satisfy $(R1)$.  Let $\{\beta_k\}_k\subseteq (0,\infty)$. Suppose that \eqref{RS} is well-defined and let $\{x_k\}_k$ be its generated sequence. Let $x^{*} \in S(f)$. Then, for any $\varepsilon > 0$:
\begin{equation*}
\exists k \leq \left\lceil \frac{b^2}{\varepsilon} \right \rceil \left(d^2(x_k,x_{k+1}) < \varepsilon\right),
\end{equation*}
where $b^2 \geq d^2(x_0, x^{*})$.
\end{lemma}
\begin{proof}
Immediately follows with Lemma \ref{easy_folklore} and Lemma \ref{algo2_fejer}.
\end{proof}

The above asymptotic estimate for $d(x_k, x_{k+1})$ can now already be used to show the main quantitative convergence theorem.

\begin{theorem}\label{secondQuantThm}
Let $X$ be a Hadamard space and let $f$ be a bifunction satisfying $(A2)$, $(A4)$ and $(A5)$. Further, let $r$ satisfy $(R1)$ and $(R2)$. Let $\{\beta_k\}_k \subseteq (0,\infty)$  be a sequence such that $\beta \leq \beta_k$ for all $k \in \N$. Suppose that \eqref{RS} is well-defined and let $\{x_k\}_k$ be its generated sequence. Suppose that $S(f) \neq \emptyset$. Then $\{x_k\}_k$ converges to the (unique) equilibrium point $x^{*}$ of $f$. Further, we have the non-asymptotic guarantee 
\begin{equation*}
\forall k > 2 \left(d(x_k, x^{*}) < \frac{D}{\sqrt{k-2}}\right)
\end{equation*}
where $D := 4b/\beta \gamma$ with $b^2 \geq d^2(x_0,x^{*})$.
\end{theorem}
\begin{proof}
It is enough to show the quantitative result, and to that end we show
\[
\forall \varepsilon > 0 \ \forall k \geq \left( \left\lceil \frac{D^2}{\varepsilon^2} \right\rceil + 1\right) (d(x_k,x^{*}) < \varepsilon).\tag{Q'}\label{secondQuant}
\]
The non-asymptotic guarantee above can then be derived similar as in Theorem \ref{THEOREM2}. Given $\varepsilon > 0$, using Lemma \ref{toolate}, choose
\begin{equation*}
	k \leq \left\lceil \frac{16 b^2}{\beta^2 \gamma^2\varepsilon^2} \right \rceil\leq\left\lceil\frac{D^2}{\varepsilon^2}\right\rceil  \text{ such that } d(x_{k+1}, x_k) < \frac{\beta\gamma}{4}\varepsilon.
\end{equation*}
Using that $x_{k+1} \in S(f_k)$, by definition we obtain
\[
0 \leq f(x_{k+1}, y) + \frac{1}{\beta_k} r(x_k,x_{k+1},y)
\]
for all $y \in X$. Consider $y := (1-\lambda) x_{k+1}\oplus \lambda x^{*}$ where $\lambda \in (0,1)$. Then
\[
0 \leq f(x_{k+1}, (1-\lambda) x_{k+1}\oplus \lambda x^{*}) + \frac{1}{\beta_k} r(x_k,x_{k+1}, (1-\lambda)x_{k+1}\oplus \lambda x^{*})
\]
By strong quasiconvexity of $f$ and $(R2)$, we obtain
\begin{align*}
0 &\leq \mathrm{max}\big\{f(x_{k+1},x^{*}), f(x_{k+1}, x_{k+1})\big\} - \lambda(1-\lambda)\frac{\gamma}{2}d^2(x^{*},x_{k+1})\\
&\hphantom{\leq \mbox{}}+\frac{1}{\beta_k}d(x_{k+1},x_k)d((1-\lambda) x_{k+1}\oplus \lambda x^{*} ,x_{k+1}).
\end{align*}
As $x^{*}$ is an equilibrium point for $f$, we know $f(x^{*},y) \geq 0$ for all $y \in X$ and thus by pseudomonotonicity of $f$, we in particular have $f(x_{k+1},x^{*}) \leq 0$. Further, $f(x_{k+1},x_{k+1}) = 0$ by $(A0)$ (which follows from $(A2)$ and $(A5)$). Thus, using convexity of $d$ in its first argument, we obtain
\[
\lambda(1-\lambda)\frac{\gamma}{2}d^2(x^{*}, x_{k+1}) \leq \frac{\lambda}{\beta_k} d(x_{k+1}, x_k)  d(x^{*},x_{k+1}).
\]
Choosing $\lambda = 1/2$ and using $\beta \leq \beta_k$, this now finally yields 
\[
\frac{\beta \gamma}{4} d^2(x^{*}, x_{k+1}) \leq d(x_{k+1},x_k)d(x^{*},x_{k+1}).
\]
If $d(x^{*}, x_{k+1}) = 0$, we are done by Lemma \ref{algo2_fejer}. Otherwise, we obtain
\[
\frac{\beta \gamma}{4} d(x^{*}, x_{k+1}) \leq d(x_{k+1},x_k) < \frac{\beta \gamma}{4} \varepsilon.
\]
As $\frac{\beta \gamma}{4} > 0$, we get $d(x^{*}, x_{k+1}) <  \varepsilon$ and again we are done with Lemma \ref{algo2_fejer}. This completes the proof of \eqref{secondQuant}.
\end{proof}

\begin{remark}\label{secondAlgStrengthening}
Already over $\mathbb{R}^d$, the above result strengthens the convergence result of Iusem and Lara (recall Theorem \ref{ILconv2}) by providing quantitative estimates, which (to our knowledge) are already novel in the Euclidean case. Again, the above approach allows for a careful examination of the main assumptions on the bifunction, whereby $(A1)$ and $(A3)$ are not needed to show the convergence result if the existence of an equilibrium point and well-definedness of the algorithm are assumed.
\end{remark}

Further, both Remark \ref{rem:Fejer} and Remark \ref{rem:Logic} extend from \eqref{PPA} to this setting of \eqref{RS}, and we again refer to the master thesis of the first author \cite{Despres2026} for additional discussions in these regards.\\

\noindent{\textbf{Acknowledgments:}} This paper is a revised version of parts of the master thesis \cite{Despres2026} of the first author, written under the supervision of Prof.\ Dr.\ Ulrich Kohlenbach at TU Darmstadt and co-supervised by the second author. Both authors want to thank Prof.\ Kohlenbach for many insightful comments on the topic of the thesis \cite{Despres2026}. The authors also want to thank Sorin-Mihai Grad and Felipe Lara for helpful conversations on and around the topic of this paper.

\bibliographystyle{plain}
\bibliography{ref}

\begin{thebibliography}{10}

\bibitem{Aleksandrov51}
A.D. Aleksandrov.
\newblock {A theorem on triangles in a metric space and some of its
  applications}.
\newblock {\em Proceedings of the Steklov Institute of Mathematics}, 38:5--23,
  1951.

\bibitem{Alexandrov57}
A.D. Aleksandrov.
\newblock {\"Uber eine Verallgemeinerung der Riemannschen Geometrie}.
\newblock {\em Schriftenreihe des Forschungsinstituts f\"ur Mathematik},
  1:33--84, 1957.

\bibitem{AnsariBabuRaju2025}
Q.H. Ansari, F.~Babu, and M.S. Raju.
\newblock {Proximal point method with Bregman distance for quasiconvex
  pseudomonotone equilibrium problems}.
\newblock {\em Optimization}, 74(9):2035--2055, 2025.

\bibitem{BauschCom2017}
H.H. Bauschke and P.L. Combettes.
\newblock {\em {Convex Analysis and Monotone Operator Theory in Hilbert
  Spaces}}.
\newblock CMS Books in Mathematics. Springer, 2nd edition, 2017.

\bibitem{Bac2013}
M.~Ba\v{c}\'ak.
\newblock {The proximal point algorithm in metric spaces}.
\newblock {\em Israel Journal of Mathematics}, 194:689--701, 2013.

\bibitem{Bacak2014}
M.~Ba\v{c}\'ak.
\newblock {\em {Convex Analysis and Optimization in Hadamard Spaces}}.
\newblock Series in Nonlinear Analysis and Applications. De Gruyter, Berlin,
  2014.

\bibitem{Bacak23}
M.~Ba\v{c}\'ak.
\newblock {Old and new challenges in Hadamard spaces}.
\newblock {\em Japanese Journal of Mathematics}, 18(2):115--168, 2023.

\bibitem{BentoCruzNetoLopesMeloFilho2024}
G.d.C. Bento, J.X.~Cruz Neto, J.O. Lopes, \'I.D.L. Melo, and P.S. Filho.
\newblock {A New Approach About Equilibrium Problems via Busemann Functions}.
\newblock {\em Journal of Optimization Theory and Applications},
  200(1):428--436, 2024.

\bibitem{BentoCruzNetoMelo2022}
G.d.C. Bento, J.X.~Cruz Neto, and \'I.D.L. Melo.
\newblock {Combinatorial Convexity in Hadamard Manifolds: Existence for
  Equilibrium Problems}.
\newblock {\em Journal of Optimization Theory and Applications},
  195(3):1087--1105, 2022.

\bibitem{BergNikolaev2008}
I.D. Berg and I.G. Nikolaev.
\newblock {Quasilinearization and curvature of Aleksandrov spaces}.
\newblock {\em Geometriae Dedicata}, 133(1):195--218, 2008.

\bibitem{Blum1994}
E.~Blum and W.~Oettli.
\newblock {From optimization and variational inequalities to equilibrium
  problems}.
\newblock {\em Mathematics Student}, 63(1):123--145, 1994.

\bibitem{BoufiFadilRoubi2026}
K.~Boufi, A.~Fadil, and A.~Roubi.
\newblock {New formulation with proximal point algorithm and proximal bundle
  methods for equilibrium problems}.
\newblock {\em Optimization}, 2026.
\newblock In press.

\bibitem{BrezisLions78}
H.~Brezis and P.L. Lions.
\newblock {Produits infinis de resolvantes}.
\newblock {\em Israel Journal of Mathematics}, 29(4):329--345, 1978.

\bibitem{BridsonHaefliger1999}
M.~Bridson and A.~Haefliger.
\newblock {\em {Metric Spaces of Non-Positive Curvature}}, volume 319 of {\em
  Grundlehren der mathematischen Wissenschaften}.
\newblock Springer, Berlin, Heidelberg, 1999.

\bibitem{BruhatTits1972}
F.~Bruhat and J.~Tits.
\newblock {Groupes R{\'e}ductifs Sur Un Corps Local}.
\newblock {\em Publications Math{\'e}matiques de l'Institut des Hautes
  {\'E}tudes Scientifiques}, 41(1):5--251, 1972.

\bibitem{BURACHIK2012}
R.~Burachik and G.~Kassay.
\newblock {On a generalized proximal point method for solving equilibrium
  problems in Banach spaces}.
\newblock {\em Nonlinear Analysis: Theory, Methods \& Applications},
  75(18):6456--6464, 2012.

\bibitem{ChaipunyaKohsakaKumam2021}
P.~Chaipunya, F.~Kohsaka, and P.~Kumam.
\newblock {Monotone vector fields and generation of nonexpansive semigroups in
  complete $\mathrm{CAT}(0)$ spaces}.
\newblock {\em Numerical Functional Analysis and Optimization},
  42(9):989--1018, 2021.

\bibitem{COLAO201261}
V.~Colao, G.~L\'opez, G.~Marino, and V.~Mart{\'\i}n-M{\'a}rquez.
\newblock {Equilibrium problems in Hadamard manifolds}.
\newblock {\em Journal of Mathematical Analysis and Applications},
  388(1):61--77, 2012.

\bibitem{CombettesHirstoaga2005}
P.L. Combettes and S.A. Hirstoaga.
\newblock {Equilibrium programming in Hilbert spaces}.
\newblock {\em Journal of Nonlinear and Convex Analysis}, 1(6):117--136, 2005.

\bibitem{Cotrina18}
J.~Cotrina and Y.~Garc{\'\i}a.
\newblock {Equilibrium Problems: Existence Results and Applications}.
\newblock {\em Set-Valued and Variational Analysis}, 26(1):159--177, 2018.

\bibitem{Despres2026}
L.M. Despr\'es.
\newblock {On the proximal point algorithm for strongly quasiconvex
  pseudomonotone equilibrium problems in Hadamard spaces}.
\newblock Master's thesis, TU Darmstadt, 2026.

\bibitem{ESPINOLA2009}
R.~Esp{\'\i}nola and A.~Fern{\'a}ndez-Le{\'o}n.
\newblock {CAT(k)-spaces, weak convergence and fixed points}.
\newblock {\em Journal of Mathematical Analysis and Applications},
  353(1):410--427, 2009.

\bibitem{Fan1961}
K.~Fan.
\newblock {A generalization of Tychonoff's fixed point theorem}.
\newblock {\em Mathematische Annalen}, 142(3):305--310, 1961.

\bibitem{Fan72}
K.~Fan.
\newblock {A Minimax Inequality and Applications}.
\newblock In O.~Sisha, editor, {\em Inequalities}, volume~3, pages 103--113.
  Academic Press, New York, 1972.

\bibitem{FerreiraOliveira2002}
O.P. Ferreira and P.R. Oliveira.
\newblock {Proximal point algorithm on Riemannian manifolds}.
\newblock {\em Optimization}, 51:257--270, 2002.

\bibitem{Flores01}
F.~Flores-Baz\'{a}n.
\newblock {Existence Theorems for Generalized Noncoercive Equilibrium Problems:
  The Quasi-Convex Case}.
\newblock {\em SIAM Journal on Optimization}, 11(3):675--690, 2001.

\bibitem{KohlenbachGerhardy07}
P.~Gerhardy and U.~Kohlenbach.
\newblock {General Logical Metatheorems for Functional Analysis}.
\newblock {\em Transactions of The American Mathematical Society},
  360:2615--2661, 2008.

\bibitem{GLM2023}
S.-M. Grad, F.~Lara, and R.T. Marcavillaca.
\newblock {Relaxed-inertial proximal point type algorithms for quasiconvex
  minimization}.
\newblock {\em Journal of Global Optimization}, 85(3):615--635, 2023.

\bibitem{GradLaraMarca24}
S.-M. Grad, F.~Lara, and R.T. Marcavillaca.
\newblock {Relaxed-Inertial Proximal Point Algorithms for Nonconvex Equilibrium
  Problems with Applications}.
\newblock {\em Journal of Optimization Theory and Applications},
  203(3):2233--2262, 2024.

\bibitem{GLM2025}
S.-M. Grad, F.~Lara, and R.T. Marcavillaca.
\newblock {Strongly Quasiconvex Functions: What We Know (So Far)}.
\newblock {\em Journal of Optimization Theory and Applications}, 205(2), 2025.
\newblock Article no.\ 38.

\bibitem{Gromov1987}
M.~Gromov.
\newblock {Hyperbolic groups}.
\newblock In S.M. Gersten, editor, {\em Essays in group theory}, volume~8 of
  {\em Mathematical Sciences Research Institute Publications}, pages 75--263.
  Springer, New York, 1987.

\bibitem{Gue1991}
O.~G\"uler.
\newblock {On the convergence of the proximal point algorithm for convex
  minimization}.
\newblock {\em SIAM Journal on Control and Optimization}, 29:403--419, 1991.

\bibitem{HadjisavvasLaraMarcavillacaVuong2026}
N.~Hadjisavvas, F.~Lara, R.T. Marcavillaca, and P.T. Vuong.
\newblock {Heavy ball and Nesterov accelerations with Hessian-driven damping
  for nonconvex optimization}.
\newblock {\em Applied Mathematics \& Optimization}, 93(3), 2026.
\newblock Article no.\ 63.

\bibitem{HieuDuongThai2021}
D.~Hieu, H.N. Duong, and B.H. Thai.
\newblock {Convergence of relaxed inertial methods for equilibrium problems}.
\newblock {\em Journal of Applied and Numerical Optimization}, 3(1):215--229,
  2021.

\bibitem{IusemKassaySosa09}
A.~Iusem, G.~Kassay, and W.~Sosa.
\newblock {On certain conditions for the existence of solutions of equilibrium
  problems}.
\newblock {\em Mathematical Programming}, 116:259--273, 2009.

\bibitem{IusemLara19}
A.~Iusem and F.~Lara.
\newblock {Optimality Conditions for Vector Equilibrium Problems with
  Applications}.
\newblock {\em Journal of Optimization Theory and Applications},
  180(1):187--206, 2019.

\bibitem{IusemLara2021}
A.~Iusem and F.~Lara.
\newblock {Proximal Point Algorithms for Quasiconvex Pseudomonotone Equilibrium
  Problems}.
\newblock {\em Journal of Optimization Theory and Applications},
  193({1--3}):443--461, 2022.

\bibitem{IusemLaraMarcavillacaYen2024}
A.~Iusem, F.~Lara, R.T. Marcavillaca, and L.H. Yen.
\newblock {A two-step proximal point algorithm for nonconvex equilibrium
  problems with applications to fractional programming}.
\newblock {\em Journal of Global Optimization}, 90(3):755--779, 2024.

\bibitem{IusemMohebbi2020}
A.~Iusem and V.~Mohebbi.
\newblock {Convergence analysis of the extragradient method for equilibrium
  problems in Hadamard spaces}.
\newblock {\em Computational and Applied Mathematics}, 39, 2020.
\newblock Article no.\ 44.

\bibitem{IusemSosa03}
A.~Iusem and W.~Sosa.
\newblock {Iterative Algorithms for Equilibrium Problems}.
\newblock {\em Optimization}, 52(3):301--316, 2003.

\bibitem{IUSEM2003}
A.~Iusem and W.~Sosa.
\newblock {New existence results for equilibrium problems}.
\newblock {\em Nonlinear Analysis: Theory, Methods \& Applications},
  52(2):621--635, 2003.

\bibitem{IusemSosa10}
A.~Iusem and W.~Sosa.
\newblock {On the proximal point method for equilibrium problems in Hilbert
  spaces}.
\newblock {\em Optimization}, 59(8):1259--1274, 2010.

\bibitem{Jost94}
J.~Jost.
\newblock {Equilibrium maps between metric spaces}.
\newblock {\em Calculus of Variations and Partial Differential Equations},
  2(2):173--204, 1994.

\bibitem{Jov1996}
M.~Jovanovi\v{c}.
\newblock {A note on strongly convex and quasiconvex functions}.
\newblock {\em Mathematical Notes}, 60:584--585, 1996.

\bibitem{Khatibzadeh2016}
H.~Khatibzadeh and V.~Mohebbi.
\newblock {Proximal point algorithm for infinite pseudo-monotone bifunctions}.
\newblock {\em Optimization}, 65(8):1629--1639, 2016.

\bibitem{KHATIBZADEHMOHEBBI2021}
H.~Khatibzadeh and V.~Mohebbi.
\newblock {Monotone and pseudo-monotone equilibrium problems in Hadamard
  spaces}.
\newblock {\em Journal of the Australian Mathematical Society},
  110(2):220--242, 2021.

\bibitem{KIRK2008}
W.A. Kirk and B.~Panyanak.
\newblock {A concept of convergence in geodesic spaces}.
\newblock {\em Nonlinear Analysis: Theory, Methods \& Applications},
  68(12):3689--3696, 2008.

\bibitem{KnasterKuratowskiMazurkiewicz1929}
B.~Knaster, C.~Kuratowski, and S.~Mazurkiewicz.
\newblock {Ein Beweis des Fixpunktsatzes f\"ur n-dimensionale Simplexes}.
\newblock {\em Fundamenta Mathematicae}, 14:132--137, 1929.

\bibitem{Kohlenbach2008}
U.~Kohlenbach.
\newblock {\em {Applied Proof Theory: Proof Interpretations and their Use in
  Mathematics}}.
\newblock Springer Monographs in Mathematics. Springer, Berlin, Heidelberg,
  2008.

\bibitem{Kohlenbach18Survey}
U.~Kohlenbach.
\newblock {Proof-theoretic Methods in Nonlinear Analysis}.
\newblock In B.~Sirakov, P.~Ney de~Souza, and M.~Viana, editors, {\em
  Proceedings of the International Congress of Mathematicians (ICM 2018)},
  volume~2, pages 61--82. World Scientific, 2019.

\bibitem{KohlenbachLN2019}
U.~Kohlenbach, G.~L{\'o}pez-Acedo, and A.~Nicolae.
\newblock {Moduli of regularity and rates of convergence for Fej{\'e}r monotone
  sequences}.
\newblock {\em Israel Journal of Mathematics}, 232(1):261--297, 2019.

\bibitem{Konnov03}
I.V. Konnov.
\newblock {Application of the proximal point method to nonmonotone equilibrium
  problems}.
\newblock {\em Journal of Optimization Theory and Applications}, 119:317--333,
  2003.

\bibitem{KristalyLiLopezAcedoNicolae2016}
A.~Krist\'aly, C.~Li, G.~L\'opez-Acedo, and A.~Nicolae.
\newblock {What do `convexities' imply on Hadamard manifolds?}
\newblock {\em Journal of Optimization Theory and Applications},
  170(3):1068--1074, 2016.

\bibitem{Lara2022}
F.~Lara.
\newblock {On Strongly Quasiconvex Functions: Existence Results and Proximal
  Point Algorithms}.
\newblock {\em Journal of Optimization Theory and Applications},
  192(3):891--911, 2022.

\bibitem{LaraMarca24}
F.~Lara and R.T. Marcavillaca.
\newblock {Bregman proximal point type algorithms for quasiconvex
  minimization}.
\newblock {\em Optimization}, 73(3):497--515, 2024.

\bibitem{LaraMarcavillacaVuong2025}
F.~Lara, R.T. Marcavillaca, and P.T. Vuong.
\newblock {Characterizations, dynamical systems and gradient methods for
  strongly quasiconvex functions}.
\newblock {\em Journal of Optimization Theory and Applications}, 206(3), 2025.
\newblock Article no.\ 60.

\bibitem{LiLopezMartinMarquez2009}
C.~Li, G.~L\'opez, and V.~Mart{\'\i}n-M{\'a}rquez.
\newblock {Monotone vector fields and the proximal point algorithm on Hadamard
  manifolds}.
\newblock {\em Journal of the London Mathematical Society}, 79:663--683, 2009.

\bibitem{Lim1976}
T.-C. Lim.
\newblock {Remarks on some fixed point theorems}.
\newblock {\em Proceedings of the American Mathematical Society}, 60:179--182,
  1976.

\bibitem{Lopez12}
R.~L\'{o}pez.
\newblock Approximations of equilibrium problems.
\newblock {\em SIAM Journal on Control and Optimization}, 50(2):1038--1070,
  2012.

\bibitem{Martinet70}
B.~Martinet.
\newblock {Br\`eve communication. {R\'egularisation} d'in\'equations
  variationnelles par approximations successives}.
\newblock {\em Revue fran\c{c}aise d'informatique et de recherche
  op\'erationnelle. S\'erie rouge}, 4(R3):154--158, 1970.

\bibitem{Moudafi99}
A.~Moudafi.
\newblock {Proximal point algorithm extended to equilibrium problems}.
\newblock {\em Journal of Natural Geometry}, 15(1--2):91--100, 1999.

\bibitem{Moudafi03}
A.~Moudafi.
\newblock {Second-order differential proximal methods for equilibrium
  problems}.
\newblock {\em Journal of Inequalities in Pure and Applied Mathematics}, 4,
  2003.
\newblock Article no.\ 18.

\bibitem{MoudafiThera99}
A.~Moudafi and M.~Th{\'e}ra.
\newblock {Proximal and Dynamical Approaches to Equilibrium Problems}.
\newblock In M.~Th{\'e}ra and R.~Tichatschke, editors, {\em Ill-posed
  Variational Problems and Regularization Techniques}, pages 187--201.
  Springer, Berlin, Heidelberg, 1999.

\bibitem{Nikolaev1995}
I.G. Nikolaev.
\newblock {The tangent cone of an Aleksandrov space of curvature $\leq K$}.
\newblock {\em Manuscripta Mathematica}, 86(1):137--147, 1995.

\bibitem{Oettli97}
W.~Oettli.
\newblock {A remark on vector-valued equilibria and generalized monotonicity}.
\newblock {\em Acta Mathematica Vietnamica}, 22:213--221, 1997.

\bibitem{Ohta2012}
S.~Ohta.
\newblock {Barycenters in Alexandrov spaces of curvature bounded below}.
\newblock {\em Advances in Geometry}, 12(4):571--587, 2012.

\bibitem{Pischke2025}
N.~Pischke.
\newblock {On the proximal point algorithm for strongly quasiconvex functions
  in Hadamard spaces}.
\newblock {\em Optimization Methods and Software}, 40(6):1438--1453, 2025.

\bibitem{Pischke2026}
N.~Pischke.
\newblock {Duality, Fr\'echet differentiability and Bregman distances in
  hyperbolic spaces}.
\newblock {\em Israel Journal of Mathematics}, 2026.
\newblock In press.

\bibitem{Polyak1966}
B.T. Polyak.
\newblock {Existence Theorems and Convergence of Minimizing Sequences in
  Extremum Problems with Restrictions}.
\newblock {\em Soviet Mathematics. Doklady}, 7:72--75, 1966.

\bibitem{Quiroz24}
E.A.~Papa Quiroz.
\newblock {Proximal Point Method for Quasiconvex Functions in Riemannian
  Manifolds}.
\newblock {\em Journal of Optimization Theory and Applications},
  202(3):1268--1285, 2024.

\bibitem{QuirzoAlexCusiMac20}
E.A.~Papa Quiroz, N.~Baygorrea Cusihuallpa, and N.~Maculan.
\newblock {Inexact Proximal Point Methods for Multiobjective Quasiconvex
  Minimization on Hadamard Manifolds}.
\newblock {\em Journal of Optimization Theory and Applications},
  186(3):879--898, 2020.

\bibitem{PapaQuirozOliveira2009}
E.A.~Papa Quiroz and P.R. Oliveira.
\newblock {Proximal point methods for quasiconvex and convex functions with
  Bregman distances on Hadamard manifolds}.
\newblock {\em Journal of Convex Analysis}, 16(1):46--69, 2009.

\bibitem{QO2012a}
E.A.~Papa Quiroz and P.R. Oliveira.
\newblock {An extension of proximal methods for quasiconvex minimization on the
  nonnegative orthant}.
\newblock {\em European Journal of Operational Research}, 216:26--32, 2012.

\bibitem{QO2012b}
E.A.~Papa Quiroz and P.R. Oliveira.
\newblock {Full convergence of the proximal point method for quasiconvex
  functions on Hadamard manifolds}.
\newblock {\em ESAIM: Control, Optimisation and Calculus of Variations},
  18(2):483--500, 2012.

\bibitem{RockTyr76}
R.T. Rockafellar.
\newblock {Monotone Operators and the Proximal Point Algorithm}.
\newblock {\em SIAM Journal on Control and Optimization}, 14(5):877--898, 1976.

\bibitem{Wald36}
A.~Wald.
\newblock {Begr\"undung einer koordinatenlosen Differentialgeometrie der
  Fl\"achen}.
\newblock {\em Ergebnisse eines Mathematischen Kolloquiums}, 7:24--46, 1936.

\end{thebibliography}

\end{document}